\documentclass[11pt]{amsart}
\usepackage{graphicx}
\usepackage{amssymb}
\usepackage{amsmath}
\usepackage{amsthm}
\usepackage{xcolor}
\usepackage{bm}
\usepackage{placeins}
\usepackage{verbatim}
\usepackage{multirow}
\usepackage{float}
\usepackage{enumitem}
\usepackage[a4paper,left=1.3in,right=1.3in,top=1.3in,bottom=1.3in]{geometry}
\usepackage{url}
\usepackage{enumitem}
\usepackage{tikz-cd}
\usepackage[all]{xy}
\usepackage{url}
\usepackage{array,booktabs}
\usepackage{algorithmicx}
\usepackage{algpseudocode}
\usepackage{algorithm}
\usepackage[hidelinks]{hyperref}
\usepackage{etex}
\usepackage[width=0.9\textwidth]{caption}
\makeatletter
\DeclareRobustCommand*{\bfseries}{%
	\not@math@alphabet\bfseries\mathbf
	\fontseries\bfdefault\selectfont
	\boldmath
}
\makeatother
\numberwithin{equation}{section} 

\newtheorem{theorem}[equation]{Theorem}
\newtheorem*{theoremM}{Theorem (McMullen)}
\newtheorem{proposition}[equation]{Proposition}
\newtheorem{lemma}[equation]{Lemma}
\newtheorem{corollary}[equation]{Corollary}

 \newtheorem*{thm:taut=twisted}{Theorem \ref{thm:taut=twisted}}
 \newtheorem*{cor:Alex=taut}{Corollary \ref{cor:Alex=taut}}
 \newtheorem*{prop:twisted_factors}{Proposition \ref{prop:twisted_factors}}
  \newtheorem*{thm:main}{Theorem \ref{thm:main}}
  
  \newtheorem*{cor:Teich:twisted}{Corollary \ref{cor:Teich:twisted}}
  \newtheorem*{prop:LN:orientable}{Corollary \ref{cor:orientable:classes}}
  \newtheorem*{prop:LN:nonorientable}{Proposition \ref{prop:no:orientable:class}}
   \newtheorem*{thm:characterisation}{Theorem \ref{thm:characterisation}}

\theoremstyle{definition}
\newtheorem{definition}[equation]{Definition}

\newtheorem{remark}[equation]{Remark}
\newtheorem{example}[equation]{Example}
\newtheorem{construction}[equation]{Construction}

\newtheorem*{remark*}{Remark}

\DeclareMathOperator{\inter}{int}

\newcommand{\ZHM}{\zz \lbrack H_M \rbrack}

\newcommand{\ZHN}{\zz \lbrack H_N \rbrack}
\newcommand{\ZH}{\zz \lbrack H \rbrack}

\newcommand{\Fit}{\mathrm{Fit}} 

\newcommand{\bez}{-}

\newcommand{\zz}{\mathbb{Z}}
\newcommand{\rr}{\mathbb{R}}
\newcommand{\V}{\mathcal{V}}
\newcommand{\bigslant}[2]{{\raisebox{.2em}{$#1$}\left/\raisebox{-.2em}{$#2$}\right.}}
\newcommand{\Mab}{M^{fab}}

\newcommand{\Vab}{\mathcal{V}^{fab}}
\newcommand{\hVab}{\hat{\mathcal{V}}^{fab}}
\newcommand{\Vuniv}{\widetilde{\mathcal{V}}}

\newcommand{\Vor}{\mathcal{V}^{or}}
\newcommand{\Mor}{M^{or}}

\newcommand{\face}[1]{\raisebox{.03em}{\large{$\mathtt{#1}$}}}

\newcommand{\Thnorm}[1]{\left\lVert #1 \right\rVert_{\mathrm{Th}}}
\newcommand{\Anorm}[1]{\left\lVert #1 \right\rVert_{\mathrm{A}}}

\algrenewcommand{\algorithmiccomment}[1]{\hspace*{\fill}
	 \color{gray}\small $\#$  #1 \color{black}\normalsize}
\begin{document}

\numberwithin{equation}{section}
\title{The taut polynomial and the Alexander polynomial}

\author{Anna Parlak}
\address{Mathematics Institute, University of Warwick, Coventry CV4 7AL, United	Kingdom\\
\indent Mathematical Institute, University of Oxford, Oxford, OX2 6GG, United Kingdom}
\email{anna.parlak@gmail.com}

\thanks{This work was supported by The Engineering and Physical Sciences Research Council (EPSRC) under grant EP/N509796/1 studentship 1936817 and by the Simons Investigator Award 409745 of Vladimir Marković.}


 \begin{abstract}
Landry, Minsky and Taylor defined the taut polynomial of a veering triangulation. Its specialisations generalise the Teichm\"uller polynomial of a fibred face of the Thurston norm ball. We prove that the taut polynomial of a veering triangulation is equal to a certain twisted Alexander polynomial of the underlying manifold. Thus the Teichm\"uller polynomials are just specialisations of twisted Alexander polynomials.
We also give formulas relating the taut polynomial and the untwisted Alexander polynomial. There are two formulas, depending on whether the maximal free abelian cover of a veering triangulation is edge-orientable or not.

Furthermore, we consider 3-manifolds obtained by Dehn filling a veering triangulation. In this case we give formulas that relate the specialisation of the taut polynomial under a Dehn filling and the Alexander polynomial of the Dehn-filled manifold. This extends a theorem of McMullen connecting the Teichm\"uller polynomial and the Alexander polynomial to the nonfibred setting, and improves it in the fibred case. We also prove a sufficient and necessary condition for the existence of an orientable fibred class in the cone over a fibred face of the Thurston norm ball.

 \end{abstract}

\maketitle%
\setcounter{tocdepth}{1}
\tableofcontents

\vspace{-1cm}
	\section{Introduction}
	Ian Agol introduced (transverse taut) \emph{veering  triangulations} as canonical triangulations of fully-punctured surgery parents of pseudo-Anosov mapping tori \cite[Section~4]{Agol_veer}. Such 3-manifolds are equipped with the \emph{suspension flow} of a pseudo-Anosov homeomorphism, and the derived veering triangulation  combinatorially encodes the dynamics of this flow. More generally, in unpublished work Agol and Gu\'eritaud showed that a veering triangulation can be derived from any \emph{pseudo-Anosov flow without perfect fits}. 	
	In particular, veering triangulations exist also on 3-manifolds that do not fibre over the circle \cite[Section 4]{veer_strict-angles}.	Conjecturally all of them come from the Agol-Gu\'eritaud's construction \cite{SchleimSegLinks}. Due to their canonicity, veering triangulations are excellent tools to study pseudo-Anosov flows combinatorially. This line of research led to new results concerning the exponential growth rate of closed orbits of semiflows on sutured manifolds \cite{LMT_flow} and Birkhoff sections for pseudo-Anosov flows~\cite{Tsang-Birkhoff}.
	
	More generally, veering triangulations have important links to the dynamics on 3-manifolds \cite{LMT_flow, LMT, SchleimSegLinks}, as well as to
	hyperbolic geometry \cite{explicit_angle, FuterTaylorWorden, Gueritaud_CT, veer_strict-angles, Worden_statistics} and the Thurston norm \cite{Landry_stable, Landry_homology_isotopy, Landry_branched}. Recently Landry, Minsky and Taylor introduced two polynomial invariants of veering triangulations --- the taut polynomial and the veering polynomial \cite{LMT}. The former can be used to generalise McMullen's Teichm\"uller polynomial \cite{McMullen_Teich} to the nonfibred setting, and is the main object of study in this paper.
	
	Our goal is to derive formulas relating the taut polynomial of a veering triangulation and a certain twisted Alexander polynomial of the underlying manifold. 	
	Throughout the paper we assume that $M$ is a connected 3-manifold equipped with a finite veering triangulation~$\mathcal{V}$.  This  implies that $M$ is a hyperbolic 3-manifold with finitely many torus cusps; see \cite[Theorem 1.5]{veer_strict-angles} for the proof of hyperbolicity and \cite[Proposition 10]{Lack_taut} for an explanation of why the links of vertices of the triangulation have to be tori.  By truncating the cusps of $M$ we obtain a 3-manifold with finitely many toroidal boundary components. We freely alternate between the cusped and the compact (truncated) model of the manifold, and we denote them both by $M$.
	Let
	\[H_M = \bigslant{H_1(M;\zz)}{\text{torsion}}.\]
 Denote by $\Mab$ the maximal free abelian cover of $M$.  The \emph{Alexander polynomial} \mbox{$\Delta_M \in \ZHM$}  of $M$ is an invariant of the homology module  $H_1(M^{fab};\zz)$; see Subsection~\ref{subsec:twisted}. More generally, in this subsection we  consider \emph{twisted Alexander polynomials} $\Delta_M^{\varphi\otimes\pi} \in \ZHM$ associated to tensor representations of $\pi_1(M)$.
 The  \emph{taut polynomial} $\Theta_{\mathcal{V}} \in \ZHM$ of $\mathcal{V}$ depends on the veering triangulation~$\Vab$ of $\Mab$ induced by $\mathcal{V}$; see Subsection \ref{subsec:taut}. 

We distinguish two types of veering triangulations: \emph{edge-orientable} ones and \emph{not edge-orientable} ones. The \emph{edge-orientation double cover}  $\mathcal{V}^{or} \rightarrow \mathcal{V}$  determines the \emph{edge-orientation homomorphism} $\omega: \pi_1(M) \rightarrow \lbrace -1, 1 \rbrace$. First we show that the taut polynomial of $\mathcal{V}$ is equal to the twisted Alexander polynomial $\Delta_M^{\omega \otimes \pi}$.
\begin{thm:taut=twisted}
Let $\mathcal{V}$ be finite a veering triangulation of a connected 3-manifold~$M$. 
Then
\[\Theta_{\mathcal{V}} = \Delta_M^{\omega \otimes \pi}\]
where $\omega: \pi_1(M) \rightarrow \lbrace -1, 1\rbrace$ is the edge-orientation homomorphism of $\mathcal{V}$ and \linebreak $\pi:\pi_1(M) \rightarrow H_M$ is the natural projection.
\end{thm:taut=twisted}

Twisted Alexander polynomials form a vast class of polynomial invariants that have been quite intesively studied; see the survey of Friedl and Vidussi \cite{FriedlVidussi_survey}. Theorem \ref{thm:taut=twisted} implies that various algebraic properties of the taut polynomial can be deduced from the algebraic properties of twisted Alexander polynomials. For instance, in Corollary~\ref{cor:symmetry} we show that the taut polynomial is \emph{symmetric}. 

Furthermore, Theorem \ref{thm:taut=twisted}  has computational consequences. It implies that the taut polynomial, like every twisted Alexander polynomial, can be computed  from the presentation of the fundamental group of the manifold using \emph{Fox calculus} \cite{FoxCalculus1, FoxCalculus5}.
An algorithm to compute the taut polynomial from its original definition was previously known~\cite{taut_veer_teich} and implemented~\cite{VeeringGitHub}. Given a veering triangulation with $n$ tetrahedra it required finding the greatest common divisor of the maximal minors of  a certain matrix of dimension $n\times (n+1)$ with coefficients in $\ZHM$ \cite[Corollary 5.7]{taut_veer_teich}. Unfortunately, computing determinants of matrices with entries in $\ZHM$ is slow, thus having to compute as many as $n+1$ determinants of $n\times n$ matrices was an obvious drawback of this algorithm.  The major advantage of the Fox calculus approach is that the fundamental group of a 3-manifold underlying a veering triangulation with $n$ tetrahedra typically has a presentation with less than $n+1$ generators. Thus the matrices one needs to work with have lower dimension.  Computation of the taut polynomial using Fox calculus has been implemented \cite{VeeringGitHub} and it indeed proved to be much faster than the original computation.

One of the open problems in the theory of veering triangulations is to find sufficient conditions for the existence of a veering triangulations on a given cusped hyperbolic 3-manifold. We still do not know, for instance, whether the Whitehead link has any veering triangulations; it is only conjectured that it does not. An interesting consequence of Theorem \ref{thm:taut=twisted} is that for any 3-manifold $M$ there is only finitely many potential candidates for the taut polynomial of a veering triangulation of $M$; see Corollary \ref{cor:finitely:many}. For the Whitehead link these would be: \begin{gather*}(ab - 1)  (ab^2 - 1), \hspace{2cm}
(ab + 1)  (ab^2 + 1), \\
(ab + 1) (ab^2 - 1), \hspace{2cm} (ab - 1)  (ab^2 + 1),\end{gather*}
where the first one is just the (untwisted) Alexander polynomial. So even though we do not know if the Whitehead link has any veering triangulation, we know what its taut polynomial can be if it has one.

It is also conjectured that any 3-manifold admits only finitely many (possibly zero) veering triangulations. This finiteness does not follow from the  result of Landry-Minsky-Taylor \cite[Theorem 5.12]{LMT} connecting veering triangulations to faces of the Thurston norm ball, because there are veering triangulations which determine  \emph{empty} faces of the Thurston norm ball (they do not carry any surfaces). Thus compactness of the Thurston norm ball is not sufficient here.
Observe that Theorem \ref{thm:taut=twisted} implies that if there is a 3-manifold~$M$ with infinitely many veering triangulations, then infinitely many of them must have the same taut polynomial. 

Finally, Theorem \ref{thm:taut=twisted} can be used to derive a relation between the taut polynomial and the (untwisted) Alexander polynomial when the triangulation $\mathcal{V}^{fab}$ of the maximal free abelian cover is edge-orientable. In this case the edge-orientation homomorphism $\omega: \pi_1(M) \rightarrow \lbrace -1, 1\rbrace$ factors through $\sigma: H_M \rightarrow \lbrace -1, 1 \rbrace$; see Lemma \ref{lem:edge-orient-factors}. Thus, by  Theorem~\ref{thm:taut=twisted}, $\Theta_{\mathcal{V}}$ differs from $\Delta_M$ only by switching the signs of variables with a nontrivial image under~$\sigma$.

\begin{cor:Alex=taut}
	Let $\mathcal{V}$ be a finite veering triangulation of a connected 3-manifold~$M$. Set $r= \mathrm{rank} \ H_M$ and let $(h_1, \ldots, h_r)$ be any basis of $H_M$. If $\Vab$ is edge-orientable then
	\[\Theta_{\mathcal{V}}(h_1, \ldots, h_r) = \Delta_M(\sigma(h_1)\cdot h_1, \ldots, \sigma(h_r)\cdot h_r),\]
	where $\sigma: H_M \rightarrow \lbrace -1, 1\rbrace$ is the factor of the edge-orientation homomorphism of $\mathcal{V}$.
\end{cor:Alex=taut}

When the triangulation of the maximal free abelian cover is not edge-orientable we relate $\Theta_{\mathcal{V}}$ and $\Delta_M$ using a polynomial invariant $\hat{\Delta}_{\V}\in \ZHM $ derived from the edge-orientation double cover of $\mathcal{V}$. In Remark \ref{remark:NEOtwisted} we explain that $\hat{\Delta}_{\V}$ can also be seen as a twisted Alexander polynomial of $M$. 
\begin{prop:twisted_factors}
	Let $\mathcal{V}$ be a finite veering triangulation of a connected 3-manifold~$M$. If $\Vab$ is not edge-orientable then
	\[\hat{\Delta}_{\V} = \Delta_M \cdot \Theta_{\mathcal{V}}.\]
	\end{prop:twisted_factors}

We then move on to analysing  a (not necessarily closed) Dehn filling~$N$ of~$M$. We set
\[H_N = \bigslant{H_1(N;\zz)}{\text{torsion}}.\] 
The inclusion of $M$ into $N$ determines an epimorphism $i_\ast: H_M \rightarrow H_N$. We compare the \emph{specialisation} $i_\ast(\Theta_{\mathcal{V}})$ of the taut polynomial of $\mathcal{V}$ \emph{under the Dehn filling} and the Alexander polynomial of $N$. In order to do that, we need to consider an intermediate free abelian cover $M^N$ of $M$ with the deck group isomorphic to $H_N$. It admits a veering triangulation $\mathcal{V}^N$. If this triangulation is edge-orientable, then the edge-orientation homomorphism factors through $\sigma_N: H_N \rightarrow \lbrace -1, 1 \rbrace$; see Lemma \ref{lem:edge-orient-factors}. Using Corollary~\ref{cor:Alex=taut} and the results of Turaev \cite[Section 4]{Turaev_Alex}, we prove the following theorem.

\begin{thm:main}
Let $\mathcal{V}$ be a finite veering triangulation of a connected 3-manifold~$M$. Let~$N$ be a Dehn filling of~$M$ such that $s=\mathrm{rank} \ H_N$ is positive. Denote by $\ell_1, \ldots, \ell_k$ the core curves of the filling solid tori in $N$ and by $i_\ast: H_M \rightarrow H_N$ the epimorphism induced by the inclusion of $M$ into $N$.	
Assume that the veering triangulation $\mathcal{V}^N$ is edge-orientable and that for every $j \in \lbrace 1, \ldots, k \rbrace$ the class $\lbrack \ell_j \rbrack \in H_N$ is nontrivial. Let \mbox{$\sigma_N: H_N \rightarrow \lbrace -1, 1 \rbrace$} be the homomorphism through which the edge-orientation homomorphism of $\mathcal{V}$  factors.
\begin{enumerate}[leftmargin=0.5cm, label={{\Roman*}}. ]
	\item $b_1(M) \geq 2$ and
	\begin{enumerate}[leftmargin=0.5cm]
		\item $s \geq 2$. Then
		\[i_\ast(\Theta_\mathcal{V})(h_1, \ldots, h_s) = \Delta_N (\sigma_N(h_1)\cdot h_1, \ldots, \sigma_N(h_s)\cdot h_s)\cdot \prod\limits_{j=1}^k (\lbrack \ell_j \rbrack - \sigma_N(\lbrack \ell_j \rbrack)).\]
		\item $s  =1$. Let $h$ denote the generator of $H_N$.
		\begin{itemize}[leftmargin=0.5cm]
			\item If $\partial N \neq \emptyset$ then 
			\[i_\ast(\Theta_\mathcal{V})(h) = (h-\sigma_N(h))^{-1}\cdot \Delta_N (\sigma_N(h)\cdot h) \prod\limits_{j=1}^k (\lbrack \ell_j \rbrack - \sigma_N(\lbrack \ell_j \rbrack)).\]
			\item If $N$ is closed then 
			\[i_\ast(\Theta_\mathcal{V})(h) = (h-\sigma_N(h))^{-2}\cdot \Delta_N(\sigma_N(h)\cdot h) \cdot \prod\limits_{j=1}^k (\lbrack \ell_j \rbrack - \sigma_N(\lbrack \ell_j \rbrack)).\]
		\end{itemize}
	\end{enumerate}
	\item $b_1(M) = 1$. Let $h$ denote the generator of $H_N \cong H_M$.
	\begin{enumerate}
		\item If $\partial N \neq \emptyset$  then
		\[i_\ast(\Theta_\mathcal{V})(h) = \Delta_N(\sigma_N(h)\cdot h).\]
		\item If $N$ is closed set $\ell = \ell_1$, and then
		\[
		i_\ast(\Theta_\mathcal{V})(h) = (h-\sigma_N(h))^{-1}(\lbrack \ell \rbrack - \sigma_N(\lbrack \ell \rbrack)) \cdot \Delta_N(\sigma_N(h)\cdot h).\]
	\end{enumerate}
\end{enumerate}

In particular  with the above assumptions we have
\[i_\ast(\Theta_{\mathcal{V}})(h_1, \ldots, h_s) = \Delta_N(\pm h_1,  \ldots, \pm h_s)\]
if and only if one of the following four conditions holds
\begin{itemize}
	\item $N=M$, or
	\item  $b_1(N)=1$, $\partial N \neq \emptyset$, $k=1$, and $\lbrack \ell_1 \rbrack$ generates $H_N$, or
	\item $b_1(N)=1$, $N$ is closed, $k=2$ and $\lbrack \ell_1 \rbrack = \lbrack \ell_2 \rbrack$ generates $H_N$, or
	\item $b_1(M)=1$, $N$ is closed and $\lbrack \ell_1 \rbrack$ generates $H_N$.
\end{itemize}
\end{thm:main}

The motivation for studying the specialisation of the taut polynomial under a Dehn filling comes from its connection with the \emph{Teichm\"uller polynomial}. This invariant is associated to a fibred face of the Thurston norm ball and was introduced by McMullen in \cite[Section~3]{McMullen_Teich}.  It plays an important role in understanding stretch factors of pseudo-Anosov homeomorphisms of surfaces. By a result of Landry, Minsky and Taylor \cite[Proposition~7.2]{LMT}, every Teichm\"uller polynomial can be expressed as a specialisation of the taut polynomial of a \emph{layered} veering triangulation. In particular, Theorem~\ref{thm:taut=twisted} implies that Teichm\"uller polynomials are just specialisations of twisted Alexander polynomials and thus can be computed using Fox calculus.
\begin{cor:Teich:twisted} Let $N$ be a connected, oriented,  hyperbolic 3-manifold. Let  $\face{F}$ be a fibred face of the Thurston norm ball in $H^1(N;\rr)$. 
	Then the Teichm\"uller polynomial~of $\face{F}$ is a specialisation of a twisted Alexander polynomial of the manifold obtained from~$N$ by drilling out the singular orbits of the suspension flow determined by $\face{F}$. 
\end{cor:Teich:twisted}
 Furthermore, by \cite[Proposition~7.2]{LMT} Theorem \ref{thm:main} can  be interpreted as a theorem which relates  the Teichm\"uller polynomial~$\Theta_{\texttt{F}}$ of a fibred face $\face{F}$ in $H^1(N;\rr)$ with the Alexander polynomial $\Delta_N$ of $N$. 
In \cite[Theorem~7.1]{McMullen_Teich} McMullen proved
\begin{theoremM}
	Let $N$ be a connected, oriented, hyperbolic 3-manifold which is fibred over the circle. Assume $b_1(N) \geq 2$. Let $\face{F}$ be a fibred face of the Thurston norm ball in $H^1(N;\rr)$. Then $\face{F} \subset \face{A}$ for a unique face $\face{A}$ of the Alexander norm ball in $H^1(N;\rr)$. If moreover the stable lamination $\mathcal{L}$ of the suspension flow determined by~$\face{F}$ is transversely orientable, then $\face{F} = \face{A}$ and $\Delta_N$ divides~$\Theta_{\mathtt{F}}$. 
\end{theoremM}

Theorem \ref{thm:main} extends McMullen's result in two ways.
\begin{enumerate}[labelindent=0pt,label=(\arabic*),itemindent=1.9em,leftmargin=0cm]
	\item When a veering triangulation is layered and Dehn filling slopes are parallel to the boundary components of a carried fibre, it gives a stronger relation between the Teichm\"uller polynomial and the Alexander polynomial. The exact formulas relating these two invariants are given in Corollary \ref{cor:Teich_in_one}. 
	
	For a fibred face $\face{F}$ whose associated lamination~$\mathcal{L}$ is transversely orientable we give sufficient conditions for equality between~$\Theta_{\texttt{F}}$ and~$\Delta_N$. When $\Delta_N$ only divides $\Theta_{\texttt{F}}$ we give the formula for the remaining factors of $\Theta_{\texttt{F}}$. We also give precise formulas relating $\Theta_{\texttt{F}}$ and $\Delta_N$ in the case when $\mathcal{L}$ is not transversely orientable; we only need to assume that its preimage in the maximal free abelian cover $N^{fab}$ of $N$ is transversely orientable. 
	In Corollary \ref{cor:faces:equal} we show that transverse orientability of this lamination in $N^{fab}$ is a sufficient condition for equality between the fibred face and a face of the Alexander norm ball.
	\item Theorem \ref{thm:main} does not require a veering triangulation to be layered. When~$\mathcal{V}$ is not layered, but does carry a surface, it determines a (not necessarily top-dimensional) nonfibred face of the Thurston norm ball \cite[Theorem 5.12]{LMT}. In this case we obtain a relation between the Alexander polynomial of a 3-manifold and a polynomial invariant of a nonfibred face of its Thurston norm ball. 
\end{enumerate} 

In Subsection \ref{subsec:orientability} we study \emph{orientable fibred classes}, that is cohomology classes determining fibrations with orientable invariant laminations in the fibre. We prove a sufficient and necessary condition for the existence of such classes in the cone over a fibred face of the Thurston norm ball.
\begin{thm:characterisation}
	Let $N$ be a compact, oriented, hyperbolic 3-manifold with a fibred face $\face{F}$ of the Thurston norm ball in $H^1(N;\rr)$. Let $\mathcal{L}$ be  the stable lamination of the suspension flow associated to $\face{F}$. If the lamination induced by $\mathcal{L}$ in the maximal free abelian cover of $N$ is transversely orientable then there is an orientable fibred class in the interior of $\rr_+ \hspace{-0.1cm}\cdot \face{F}$.
\end{thm:characterisation}
In other words, the assumptions of Corollary \ref{cor:Teich_in_one} are equivalent to the existence of orientable fibred classes in the cone over a fibred face $\face{F}$. Since such classes determine fibrations for which the stretch factor and the  homological stretch factor of the monodromy are equal, they are the reason behind the close relationship between the Teichm\"uller polynomial and the Alexander polynomial obtained in Corollary \ref{cor:Teich_in_one}. 

The ``if'' part of Theorem \ref{thm:characterisation} is constructive and appears as Corollary \ref{cor:orientable:classes}.

\begin{prop:LN:orientable}
	Let~$\face{F}$ be a fibred face of the Thurston norm ball in $H^1(N;\rr)$. Let  $(h_1, \ldots, h_s)$ be a basis of $H_N$ and let $(h_1^\ast, \ldots, h_s^\ast)$ be the dual basis of $H^1(N;\zz)$. Denote by $\mathcal{L}$ the stable lamination of the suspension flow associated to $\face{F}$.
	
	If $\mathcal{L}$ is not transversely orientable, but the induced lamination $\mathcal{L}^N \subset N^{fab}$ is transversely orientable,
	then a primitive class
	\[\xi=a_1h_1^\ast + \ldots + a_r h_r^\ast \in \inter(\rr_+\hspace{-0.1cm}\cdot \face{F}) \cap H^1(N;\zz)\]
	is orientable if and only if for every $j=1, \ldots, s$
	\begin{equation}\label{eqn:in:intro}
		a_j = \begin{cases}
			\text{odd} &\text{if }\sigma_N(h_j) = -1\\
			\text{even} &\text{if }\sigma_N(h_j) = 1.
		\end{cases}
	\end{equation} In particular, with the above assumptions on $\mathcal{L}$ there are orientable fibred classes in $\rr_+ \hspace{-0.1cm}\cdot \face{F}$. \qed
\end{prop:LN:orientable}

\subsection*{Acknowledgements}
I am grateful to Samuel Taylor for explaining to me his work on the taut and veering polynomials during my visit at Temple University in July 2019, and subsequent conversations. I thank Saul Schleimer and Henry Segerman for their generous help in implementing the computation of the taut polynomial. I thank Dawid Kielak and Bin Sun for discussions on twisted Alexander polynomials, and Joe Scull for discussions on 3-manifolds.

The majority of this work was written during PhD studies of the author at the University of Warwick, funded by the EPSRC and supervised by Saul Schleimer. However, the proof that the taut polynomial is always a twisted Alexander polynomial was added when the author was a Postdoctoral Research Associate at the University of Oxford, funded by the Simons Investigator Award 409745 of Vladimir Marković.
\section{Transverse taut veering triangulations}\label{sec:veering}
	\subsection*{Ideal triangulations of 3-manifolds}\label{subsec:triangulations}
An \emph{ideal triangulation} of a 3-manifold~$M$ with boundary is a decomposition of the interior of $M$ into ideal tetrahedra. We denote an ideal triangulation by $\mathcal{T} = (T,F, E)$, where $T, F, E$ denote the set of tetrahedra, triangles (2-dimensional faces) and edges, respectively. 
\subsection*{Transverse taut triangulations}
Let $t$ be an ideal tetrahedron. Assume that a coorientation is assigned to each of its faces. We say that $t$ is \emph{transverse taut} if on two of its faces the coorientation points into $t$ and on the other two it points out of $t$ \cite[Definition 1.2]{veer_strict-angles}. We call the pair of faces whose coorientations point out of $t$ the \emph{top faces} of~$t$ and the pair of faces whose coorientations point into $t$ the \emph{bottom faces} of~$t$. We also say that $t$ is \emph{immediately below} its top faces and \emph{immediately above} its bottom faces. We encode a transverse taut structure on a tetrahedron by drawing it as a quadrilateral with two diagonals --- one on top of the other; see Figure~\ref{fig:veering_tetra}. The convention is that the coorientation on all faces points towards the reader. 

The \emph{top diagonal} is the common edge of the two top faces of $t$ and the \emph{bottom diagonal} is the common edge of the two bottom faces of $t$. The remaining four edges of $t$ are called its \emph{equatorial edges}. Presenting a transverse tetrahedron as in Figure~\ref{fig:veering_tetra} naturally endows it with an abstract assignment of angles from $\lbrace 0, \pi\rbrace$ to its edges. The diagonal edges are assigned the $\pi$ angle, and the equatorial edges are assigned the $0$ angle. Such an assignment of angles is called a \emph{taut angle structure} on $t$ \cite[Definition~1.1]{veer_strict-angles}.

A triangulation $\mathcal{T}=(T,F, E)$ is \emph{transverse taut} if 
\begin{itemize}
	\item every ideal triangle $f \in F$ is assigned a coorientation so that every ideal tetrahedron is transverse taut, and 
	\item for every edge $e \in E$ the sum of angles of the underlying taut angle structure of $\mathcal{T}$, over all embeddings of $e$ into tetrahedra, equals $2\pi$ \cite[Definition 1.2]{veer_strict-angles}. 
	\end{itemize}
We denote a triangulation with a transverse taut structure by $(\mathcal{T}, \alpha)$.	
\begin{remark*} Transverse taut triangulations  were introduced by Lackenby in~\cite{Lack_taut} under the name of \emph{taut triangulations}.  \end{remark*}
\subsection*{Veering triangulations} 	A transverse taut tetrahedron $t$ is \emph{veering} if its edges are coloured either red or blue so that the pattern of colouring on the equatorial edges of~$t$ is precisely as in Figure~\ref{fig:veering_tetra}. There is no restriction on how the diagonal edges of $t$ are coloured. This is indicated by colouring them black.
 A transverse taut triangulation  is \emph{veering} if a colour (red/blue) is assigned to each edge of the triangulation so that every tetrahedron is veering. We denote a  veering triangulation by $\mathcal{V}= (\mathcal{T}, \alpha, \nu)$, where $\nu$ corresponds to the colouring on edges.

\begin{figure}[h]
	\begin{center}
		\includegraphics[width=0.21\textwidth]{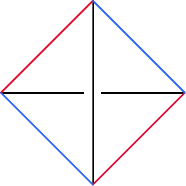}
		\put(-75,67){R}
		\put(-20,67){B}
		\put(-20,12){R}
		\put(-75,12){B}
	\end{center}
	\caption{A veering tetrahedron. Red edges are labelled with `R' and blue edges are labelled with `B'. Unlabelled edges can be of either colour. The taut angle structure assigns 0 to the equatorial edges and $\pi$ to the diagonal edges of the tetrahedron.}
	\label{fig:veering_tetra} 
\end{figure}

	
	\begin{remark}[\textbf{Reading pictures of veering triangulations}] \label{remark:pictures}
		Figures \ref{fig:m003} and \ref{fig:m003_cover} show all tetrahedra of some veering triangulations. In each of these figures, each column represents one tetrahedron of the triangulation. The top row of each column presents the top faces of the tetrahedron, and the bottom row presents the  bottom faces of the tetrahedron. Each triangle has a label of the form $f_i$ (or $\overline{f_i}$ in the case of Figure~\ref{fig:m003_cover}). Since a veering triangulation is in particular transverse taut, each face of the  triangulation is a top face of some tetrahedron, and  a bottom face of some tetrahedron. This means that  the label $f_i$ appears exactly once in the top row  and exactly once in the bottom row. In the triangulation, the two copies of $f_i$ are identified by a homeomorphism which sends (ideal) edges to (ideal) edges.  It therefore suffices to know the correspondence between the edges of the top copy of $f_i$ and edges of the bottom copy of $f_i$. This can be read off from the train track that we drew in the triangles of the triangulation (this is the upper track of the triangulation, as defined in Definition \ref{defn:tracks}). Namely, under the identification the train track in the bottom copy of $f_i$ must agree with the train track in the top copy of $f_i$.
		
		To avoid cluttering Figures \ref{fig:m003} and \ref{fig:m003_cover} with too many labels, we do not indicate the colours of edges by labels (we continue to use colours though). In case of black-and-white printing the colours of edges can be deduced using the following rules:
		\begin{itemize}
			\item if an edge has positive slope then it is red,
			\item if an edge has negative slope then it is blue,
			\item if an edge is either horizontal or vertical then its colour can be deduced using identifications between the faces of tetrahedra of the triangulation.
		\end{itemize}
	\end{remark}

The importance of veering triangulations follows from their connection with pseudo-Anosov flows on 3-manifolds. Agol-Gu\'eritaud showed that given a pseudo-Anosov flow without perfect fits on a closed 3-manifold one can always construct a veering triangulation on the complement of the singular orbits of the flow. This result is unpublished but its outline can be found in \cite[Section 4]{LMT_flow}. The constructed veering triangulation is canonical and combinatorially encodes the stable and unstable laminations of the flow; see Lemma \ref{lemma:EO:to}. Recently Agol-Tsang proved that the converse is also true. That is, given a veering triangulation of $M$ it is possible to construct a pseudo-Anosov flow on certain Dehn fillings of $M$ \cite[Theorem 5.1]{Tsang-Agol}. Thus veering triangulations are combinatorial tools that can be used to study pseudo-Anosov flows. In particular, the main object of interest in this paper, the taut polynomial, carries information about the exponential growth rate of closed orbits of the associated flow; see \cite{LMT_flow, LMT}.

\section{The lower and upper tracks}\label{sec:train:tracks}
Let $(\mathcal{T}, \alpha)$ be a transverse taut triangulation of $M$.  The taut structure on~$\mathcal{T}$ allow us to view its 2-skeleton $\mathcal{T}^{(2)}$ as a \emph{branched surface} embedded in $M$. We call it the \emph{horizontal branched surface} of $(\mathcal{T}, \alpha)$ and denote it by $\mathcal{B}$ \cite[ Subsection 2.12]{SchleimSegLinks}. 

In this subsection we describe a pair of canonical train tracks embedded in $\mathcal{B}$, called the \emph{lower} and \emph{upper tracks} of a $(\mathcal{T}, \alpha)$. They are key in the definition of the taut polynomial. Moreover, we use them to divide veering triangulations into two classes: \emph{edge-orientable} ones  and \emph{not edge-orientable} ones.
\subsection{Dual train tracks}
 A \emph{dual train track} $\tau \subset \mathcal{B}$ is a train track in $\mathcal{B}$ which restricted to any face $f \in F$ looks like in Figure~\ref{fig:triangle_train_track}. We denote the restriction of $\tau$ to $f$ by $\tau_f$. The edges of $\tau_f$ are called the \emph{half-branches} of $\tau_f$. Two of the half-branches of $\tau_f$ are called its \emph{small} half-branches. The remaining half-branch is called \emph{large}.

The train track $\tau$ has two types of special points: \emph{switches}  (one for each $f \in F$; the common point of all half-branches of $\tau_f$) and \emph{midpoints of edges} (one for each $e \in E$). We say that a half-branch $b$ \emph{meets} $e \in E$, or vice versa, if one of the endpoints of $b$ is the midpoint of $e$.

\begin{figure}[h]
	\begin{center}
		\includegraphics[width=0.13\textwidth]{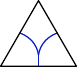} 
		\put(-37,-12){large}
		\put(-70,22){small}
		\put(-8,22){small}
	\end{center}
	\caption{Dual train track restricted to a face.}
	\label{fig:triangle_train_track}
\end{figure}

\begin{definition}
	We say that a dual train track $\tau \subset \mathcal{B}$ is \emph{transversely orientable} if there exists a nonvanishing continuous vector field on $\tau$ which is tangent to $\mathcal{B}$ and transverse to $\tau$. 
\end{definition}
\begin{figure}[h]
	\begin{center}
		\includegraphics[width=0.35\textwidth]{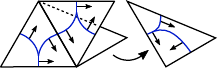}
	\end{center}
	\caption{A local picture of a transversely oriented dual train track in the horizontal branched surface. For clarity the train track in the bottom triangle is presented separately on the right.}
	\label{fig:transversely_oriented_track}
\end{figure}
Figure \ref{fig:transversely_oriented_track} presents a local picture of a transversely oriented dual train track. Transverse orientation is indicated  by arrows transverse to the branches of the track pointing in one of the two possible directions.
\begin{example}\label{example:m003_upper_track}
The SnapPy manifold m003 \cite{snappea}, also known as the figure-eight knot sister, admits a veering triangulation described in the Veering Census \cite{VeeringCensus} by the string \texttt{cPcbbbdxm\_10}. Figure \ref{fig:m003} presents an example of a dual train track $\tau$ in the horizontal branched surface of \texttt{cPcbbbdxm\_10}. 
\begin{figure}[h]
	\begin{center}
		\includegraphics[width=0.35\textwidth]{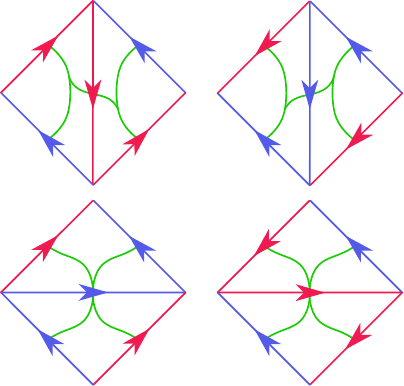} 
		\put(-135,102){$f_1$}
		\put(-95,102){$f_3$}
		\put(-58,102){$f_2$}
		\put(-18,102){$f_0$}
		\put(-115,52){$f_0$}
		\put(-38,52){$f_3$}
		\put(-115,10){$f_2$}
		\put(-38,10){$f_1$}
	\end{center}
	\caption{The upper track of  the veering triangulation \texttt{cPcbbbdxm\_10} of the figure-eight knot sister. See Remark \ref{remark:pictures} for an explanation of how to read this picture.   Observe that the yellow shaded region is homeomorphic to the M\"obius band.}
	\label{fig:m003}
\end{figure}

The union of two half-branches of $\tau$ in the face $f_0$  form the core curve of the M\"obius band (shaded in yellow). Hence the presented train track cannot be transversely oriented.
\end{example}
\subsection{The lower and upper tracks}
A transverse taut structure $\alpha$  on $\mathcal{T}$ endows its horizontal branched surface $\mathcal{B}$ with a pair of canonical train tracks which we call, following \cite[Definition 4.7]{SchleimSegLinks}, the \emph{lower} and \emph{upper}  tracks of $(\mathcal{T}, \alpha)$. 

\begin{definition}\label{defn:tracks}Let $(\mathcal{T}, \alpha)$ be a transverse taut triangulation and~ let $\mathcal{B}$ be the corresponding horizontal branched surface. 	
	The \emph{lower track} $\tau^L$ of $\mathcal{T}$ is the dual train track in $\mathcal{B}$ such that for every $f \in F$ the large-half branch of $\tau^L_f$ meets the top diagonal of the tetrahedron immediately below~$f$.	
	The \emph{upper track} $\tau^U$ of $\mathcal{T}$ is the dual train track in $\mathcal{B}$ such that for every $f \in F$ the large-half branch of $\tau^U_f$ meets the bottom diagonal of the tetrahedron immediately above~$f$.
\end{definition}
While this definition may seem artificial at first, in the veering case the lower and upper tracks combinatorially encode the unstable and stable laminations of the associated pseudo-Anosov flow; see Subsection \ref{subsec:EO:to}.
\begin{example}The upper track of the veering triangulation \texttt{cPcbbbdxm\_10} of the figure-eight knot sister is presented in Figure \ref{fig:m003}.
	\end{example}

We introduce the following names for the edges of $f \in F$ which meet large half-branches of $\tau_f^L$ or $\tau_f^U$.
\begin{definition}
	Let $(\mathcal{T}, \alpha)$ be a transverse taut triangulation. We say that an edge in the boundary of $f \in F$ is the \emph{lower large} (respectively the \emph{upper large}) edge of $f$ if it meets the large half-branch of $\tau_{f}^L$ (respectively $\tau_{f}^U$).
\end{definition}
To define the lower and upper tracks we do not need a veering structure on a transverse taut triangulation. However, if the triangulation is veering then the lower and upper tracks restricted to the faces of a single tetrahedron $t$ are determined by the colours of the diagonal edges of $t$; see the lemma below. 
\begin{lemma}[Lemma 3.9 of \cite{taut_veer_teich}]\label{lemma:large_edges}
	Let $\mathcal{V}$ be a veering triangulation. Let $t$ be one of its tetrahedra. The lower large edges of the bottom faces of $t$ are the equatorial edges of~$t$ which are of the same colour as the bottom diagonal of~$t$. The upper large edges of the top faces of $t$ are the equatorial edges of $t$ which are of the same colour as the top diagonal of $t$. \qed
\end{lemma} 
Using Lemma \ref{lemma:large_edges} we present the pictures of the lower track and upper track in a veering tetrahedron  in Figure~\ref{fig:lower_tt}(a) and Figure~\ref{fig:lower_tt}(b),  respectively. 
 \begin{figure}[h]
	\begin{center}
		\includegraphics[width=0.35\textwidth]{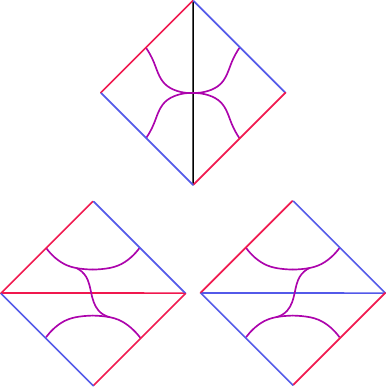} 
		\includegraphics[width=0.35\textwidth]{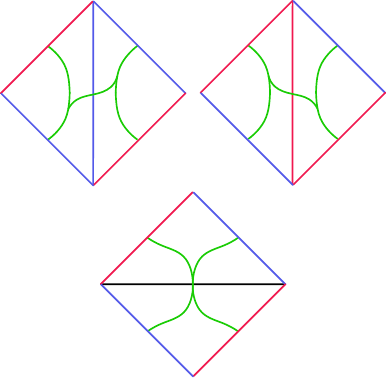} 
		\put(-300,-10){(a)}
		\put(-140,-10){(b)}
		\put(-248,128){R}
			\put(-199,128){B}
				\put(-199,82){R}
					\put(-246,82){B}
		\put(-137,125){R}
		\put(-90,125){B}
		\put(-63,125){R}
		\put(-15,125){B}
		\put(-135,79){B}
		\put(-89,79){R}
		\put(-61,79){B}
		\put(-15,79){R}
		\put(-118,118){B}
			\put(-33,118){R}
		\put(-285,53){R}
		\put(-236,53){B}
		\put(-283,8){B}
		\put(-236,8){R}
		\put(-210,53){R}
		\put(-162,53){B}
		\put(-208,8){B}
		\put(-162,8){R}
		\put(-243,37){R}
		\put(-203,37){B}
		\put(-100,53){R}
		\put(-53,53){B}
		\put(-98,8){B}
		\put(-52,8){R}

	\end{center}
	\caption{Squares in the top row represent top faces of a tetrahedron and squares in the bottom row represent bottom faces of a tetrahedron. (a) The lower track in a veering tetrahedron. 
	There are two options, depending on the colour of the bottom diagonal. 
		(b) The upper track in a veering tetrahedron. 
	There are two options, depending on the colour of the top diagonal.}
	\label{fig:lower_tt}
\end{figure}
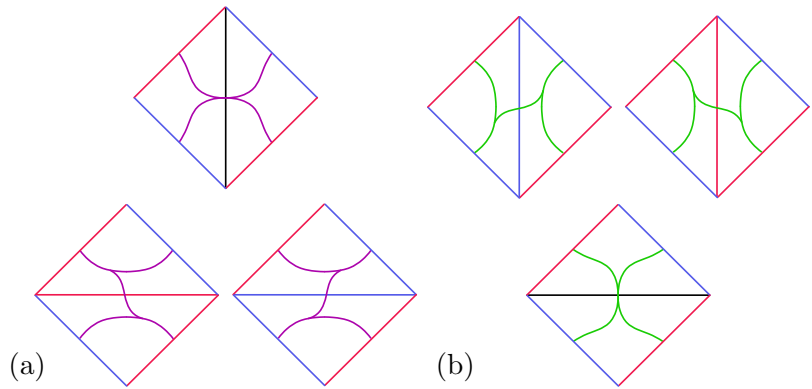
\subsection{Edge-orientable veering triangulations}\label{subsec:EO}
We divide all veering triangulations into two classes. The division depends on whether the lower/upper track is transversely orientable. First we prove that it does not matter which track we consider. 
\begin{lemma}\label{lemma:edge_orientable}
	Let $\mathcal{V}$ be a veering triangulation. The lower track of $\mathcal{V}$ is transversely orientable if and only if the upper track of $\mathcal{V}$ is transversely orientable.
\end{lemma}
\begin{proof}
	Suppose $\tau^L$ is transversely oriented. Pick the orientation on the edges of $\mathcal{V}$ determined by the transverse orientation on $\tau^L$.	
	Note that for any triangle $f \in F$ the lower large edge of $f$ and the upper large edge of~$f$ are different edges of $f$, but both are of the same colour as $f$ \cite[Lemma 2.1]{taut_veer_teich}. Figure \ref{fig:defn_edge-orientable} presents the lower and upper tracks in a veering triangle.	
	If we reverse the orientation of all edges of one colour, say blue, then the new orientations determine a transverse orientation on~$\tau^U$. 	
\end{proof}
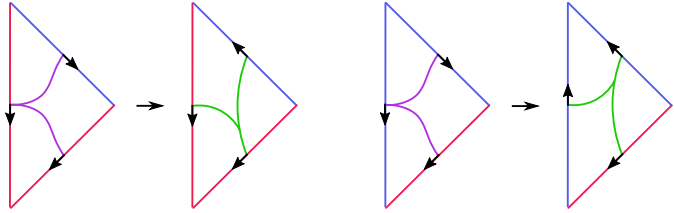
\begin{figure}[h]
	\begin{center}
		\includegraphics[width=0.61\textwidth]{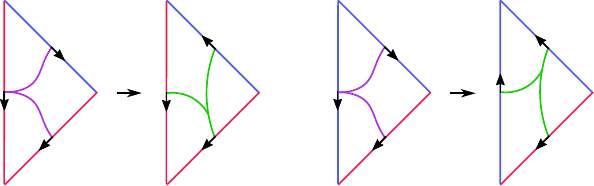} 
		\put(-260,55){R}
		\put(-228,61){B}
		\put(-226,11){R}
		\put(-191,55){R}
		\put(-158,61){B}
		\put(-158,11){R}
		\put(-118,55){B}
		\put(-86,61){B}
		\put(-86,11){R}
		\put(-50,55){B}
		\put(-17,61){B}
		\put(-17,11){R}
	\end{center}
	\caption{Using a transverse orientation on the lower track to find a transverse orientation on the upper track. On the left: in a  triangle with two red edges. On the right: in a triangle with two blue edges.}
	\label{fig:defn_edge-orientable}
\end{figure}
\begin{definition}
	We say that a veering triangulation $\mathcal{V}$ is \emph{edge-orientable} if the upper track of $\mathcal{V}$ is transversely orientable.
\end{definition}
The name ``edge-orientable'' was suggested by Sam Taylor. It reflects the observation that restricting a transverse orientation on a dual train track to the edge midpoints endows the edges of the triangulation with an orientation.

By Lemma \ref{lemma:edge_orientable} a veering triangulation $\mathcal{V}$ is edge-orientable if and only if the lower track of $\mathcal{V}$ is transversely orientable. For the remainder of the paper we consider only the upper track and use a shorter notation $\tau = \tau^U$.
\begin{example}\label{ex:m003_neo}
	We observed in Example \ref{example:m003_upper_track} that the upper track of the veering triangulation \texttt{cPcbbbdxm\_10} of the figure-eight knot sister is not transversely orientable. Hence this triangulation is not edge-orientable. 
\end{example}
\subsection{Edge-orientability and the stable lamination of the underlying flow}\label{subsec:EO:to}
The upper track restricted to the two bottom faces of a tetrahedron $t$ differs from the upper track restricted to the two top faces of $t$ by a splitting; see Figure~\ref{fig:lower_tt}. Performing this split continuously sweeps out a 2-dimensional complex in $t$ which can be transformed into a branched surface $\mathcal{B}^U_t$; see Figure~6.2.(B) of \cite{SchleimSegLinks}. The union $\mathcal{B}^U = \bigcup\limits_{t\in T} \mathcal{B}^U_t$ forms a branched surface in $M$, called the \emph{upper branched surface} \cite[Subsection 6.1]{SchleimSegLinks}. 

Recall that a pseudo-Anosov flow determines a pair of 2-dimensional laminations in the manifold, the \emph{stable} and \emph{unstable} one \cite[Definition 7.1]{Fenley_geometry}.
It follows from the Agol-Gu\'eritaud's construction that if a veering triangulation~$\mathcal{V}$ comes from a pseudo-Anosov flow $\Psi$ then the stable lamination $\mathcal{L}$ of~$\Psi$ is fully carried by  $\mathcal{B}^U$; see \cite[Section~4]{LMT_flow} and \cite[Theorem 8.1]{SchleimSegLinks}. Therefore edge-orientability of $\mathcal{V}$ is equivalent to transverse orientability of $\mathcal{L}$.
\begin{lemma}\label{lemma:EO:to}
	Suppose that a veering triangulation $\mathcal{V}$ is constructed from a pseudo-Anosov flow $\Psi$. Let $\mathcal{L}$ denote the stable lamination of $\Psi$. Then $\mathcal{V}$ is edge-orientable if and only if $\mathcal{L}$ is transversely orientable. \qed
\end{lemma}

\FloatBarrier

\section{The edge-orientation double cover}\label{sec:EO_cover}
Let $\mathcal{V}$ be a finite veering triangulation of a connected 3-manifold $M$, with the set $T$ of tetrahedra, the set $F$ of faces and the set $E$ of edges. Let $\tau = \tau^U$ be the upper track of $\mathcal{V}$. In this section we construct the \emph{edge-orientation double cover} $\Vor$ of $\mathcal{V}$. 

Let $f \in F$. The train track $\tau_f$ is transversely orientable; if we fix a transverse orientation on one half-branch of $\tau_f$ there is no obstruction to extend it over all half-branches of $\tau_f$. After fixing a transverse orientation on $\tau_f$ we say that $f$ is \emph{edge-oriented} if the orientation of the edges in the boundary of $f$ agrees with the orientation determined by the transverse orientation on $\tau_f$; see Figure \ref{fig:edge_oriented}.
\begin{figure}[h]
	\begin{center}
		\includegraphics[width=0.5\textwidth]{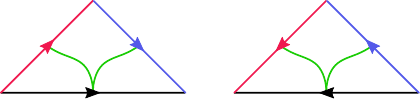} 
		\put(-205,15){R}
		\put(-122,15){B}
		\put(-90,15){R}
		\put(-9,15){B}
	\end{center}
	\caption{Edge-oriented triangles. There are two possibilities depending on the chosen transverse orientation on $\tau_f$.}
	\label{fig:edge_oriented}
\end{figure}
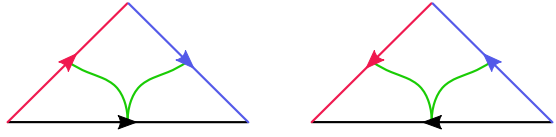

We are interested in extending the transverse orientation of $\tau_f$ further, not necessarily over the whole 2-skeleton of the triangulation (which might not be possible; see Example \ref{ex:m003_neo}), but over the four faces of a given tetrahedron. 
\begin{lemma}\label{lem:edge-oriented_tetrahedra}
	Let $\mathcal{V}$ be a veering triangulation. Fix a tetrahedron $t$ of $\mathcal{V}$ and let $f_i$, $i=1,2,3,4$, be the four faces of $t$. Denote by $\tau_i$  the restriction $\tau^U_{f_i}$ of the upper track~$\tau^U$ of $\mathcal{V}$ to $f_i$. Then the train track \[\tau_t = \bigcup\limits_{i=1}^4 \tau_i \subset\bigcup\limits_{i=1}^4 f_i\] is transversely orientable. 	
\end{lemma}
\begin{proof} 
	We can transport a fixed transverse orientation from $\tau_1$ to $\tau_i$, $i=2,3,4$, across the edges of $f_1$ without twisting. Using Figure~\ref{fig:lower_tt} one can check that the obtained transverse orientations on $\tau_i$'s agree along every edge of $t$.	
\end{proof}

\begin{definition}
	Let $t \in T$. Fix a transverse orientation on the upper track $\tau_t$ restricted to the faces of $t$; this is possible by Lemma \ref{lem:edge-oriented_tetrahedra}. We say that $t$ is \emph{edge-oriented} if the orientation on the edges of $t$ agrees with the orientation determined by the transverse orientation on $\tau_t$. In that case every face of $t$ is edge-oriented.
\end{definition}
We construct the edge-orientation double cover $\Vor$ of $\mathcal{V}$ using edge-oriented tetrahedra --- twice as many as the number of tetrahedra of $\mathcal{V}$. Then we use the fact that every face can be edge-oriented in precisely two ways (see Figure \ref{fig:edge_oriented}) to argue that the faces of the edge-oriented tetrahedra can be identified in pairs to give a double cover of $\mathcal{V}$.
\begin{construction}(Edge-orientation double cover)\label{EOconstruction}
Let $\mathcal{V}$ be a veering triangulation with the set $T$ of tetrahedra, the set $F$ of triangular faces, and the set $E$ of edges. For $t \in T$ we take two copies of $t$ and denote them by $t, \overline{t}$. Both $t, \overline{t}$ are endowed with the upper track of $\mathcal{V}$ restricted to the faces of $t$. We denote the track in~$t$ by~$\tau_t$ and the track in $\overline{t}$ by $\overline{\tau}_t$.

We fix an orientation on the bottom diagonal  of $t$ and choose the \emph{opposite} orientation on the bottom diagonal of $\overline{t}$. They determine (opposite) transverse orientations on ~$\tau_t$,~$\overline{\tau}_t$; see Lemma \ref{lem:edge-oriented_tetrahedra}. 
We orient the edges of $t, \overline{t}$ so that they are edge-oriented by~$\tau_t$,~$\overline{\tau}_t$, respectively. 

Let $\overline{T} = \lbrace \overline{t} \ | \ t \in T\rbrace$. Every face $f \in F$ appears four times as a face of an edge-oriented tetrahedron from $T \cup \overline{T}$: twice as a bottom face and twice as a top face. Moreover, its two bottom copies are edge-oriented in the opposite way; the same is true for its two top copies. Hence we can identify the faces of the tetrahedra $T \cup \overline{T}$ in pairs to obtain an edge-oriented veering triangulation which double covers~$\mathcal{V}$.
\end{construction}
\begin{definition}
	Let $\mathcal{V}$ be a veering triangulation. The \emph{edge-orientation double cover} $\Vor$ of $\mathcal{V}$ is the ideal triangulation obtained by Construction \ref{EOconstruction}. 
\end{definition}
\begin{example}
	Let $\mathcal{V}$ denote the veering triangulation \texttt{cPcbbbdxm\_10} of the figure-eight knot sister (manifold m003). It is presented in Figure \ref{fig:m003}. 
	We follow Construction~\ref{EOconstruction} to build the edge-orientation double cover $\Vor$ of $\mathcal{V}$. The result is presented in Figure~\ref{fig:m003_cover}. 
	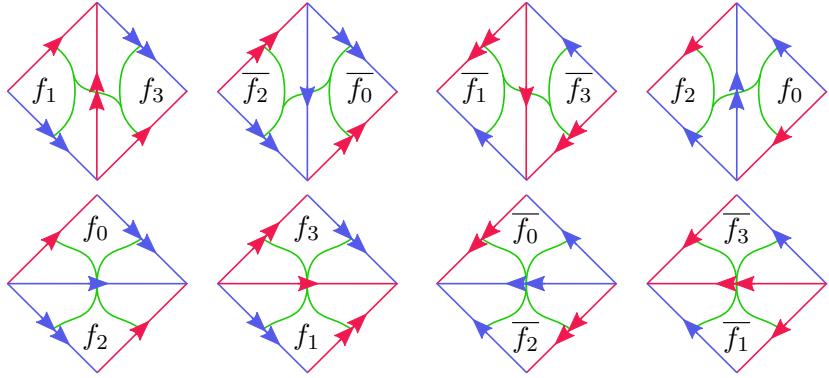
\begin{figure}[h]
		\begin{center}
			\includegraphics[width=0.75\textwidth]{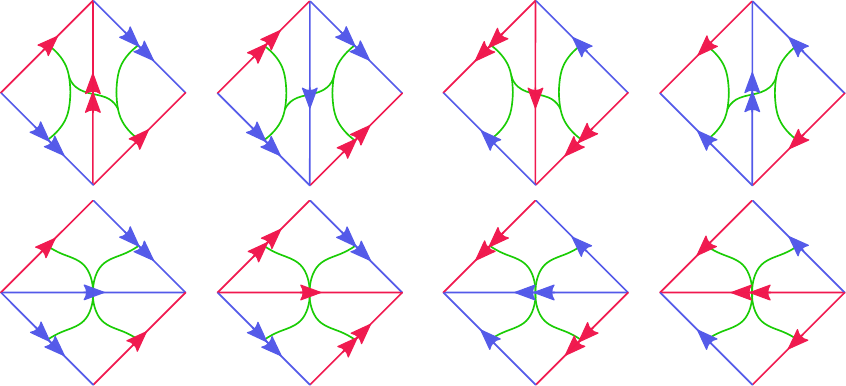} 
			\put(-298,104){$f_1$}
			\put(-258,104){$f_3$}
			\put(-219,104){$\overline{f_2}$}
			\put(-180,104){$\overline{f_0}$}
			\put(-279,52){$f_0$}
			\put(-200,52){$f_3$}
			\put(-279,11){$f_2$}
			\put(-200,11){$f_1$}
			\put(-137,104){$\overline{f_1}$}
			\put(-98,104){$\overline{f_3}$}
			\put(-59,104){$f_2$}
			\put(-19,104){$f_0$}
			\put(-118,51){$\overline{f_0}$}
			\put(-39,51){$\overline{f_3}$}
			\put(-118,10){$\overline{f_2}$}
			\put(-39,10){$\overline{f_1}$}
		\end{center}
		\caption{The edge-orientation double cover of the veering triangulation  $\mathcal{V}$ = \texttt{cPcbbbdxm\_10} of the figure-eight knot sister. See Remark \ref{remark:pictures} for an explanation of how to read this figure. Two lifts of an edge of~$\mathcal{V}$ have the same colour, but the orientation on one of them is indicated by a single arrow, while on the other --- by a double arrow.}
		\label{fig:m003_cover}
	\end{figure}
\end{example}

Since we build the cover $\Vor$ using edge-oriented tetrahedra, we immediately get the following lemma.
\begin{lemma}\label{lem:Vor_is_EO}
	Let $\mathcal{V}$ be a veering triangulation. The edge-orientation double cover $\Vor$ of $\mathcal{V}$ is edge-orientable. \qed
\end{lemma}

\begin{lemma}
	Let $\mathcal{V}$ be a veering triangulation. The edge-orientation double cover $\Vor$ is connected if and only if $\mathcal{V}$ is not edge-orientable.
\end{lemma}
\begin{proof}
	Suppose that $\mathcal{V}$ is edge-orientable. Fix the orientations on the edges of $\mathcal{V}$ such that every face of $\mathcal{V}$ is edge-oriented. With this choice of orientations following Construction \ref{EOconstruction} gives a disjoint union of two copies of $\mathcal{V}$.
	Conversely, if $\Vor$ is disconnected, then it is a disjoint union of two copies of $\mathcal{V}$. By Lemma \ref{lem:Vor_is_EO} the triangulation $\Vor$ is edge-orientable and hence so is $\mathcal{V}$.
\end{proof}
\subsection{The edge-orientation homomorphism}\label{subsec:EO_homo}
The edge-orientation double cover $\Vor \rightarrow \mathcal{V}$   determines  a homomorphism 
\begin{gather}\label{eqn:edge-orientation_homo}
	\omega: \pi_1(M) \rightarrow \lbrace -1, 1\rbrace \nonumber \\
	\gamma \mapsto \begin{cases} 1 &\text{if $\gamma$ lifts to a loop in $\mathcal{V}^{or}$}\\
		-1 &\text{otherwise}.\end{cases}\end{gather}
We call this homomorphism the \emph{edge-orientation homomorphism} of $\mathcal{V}$. 

Let $H$ be a quotient of $\pi_1(M)$ with the corresponding quotient map \mbox{$q: \pi_1(M) \rightarrow H$}.
We are interested in a combinatorial condition which ensures that the edge-orientation homomorphism $\omega$ factors through~$q$. The remaining factor \mbox{$H \rightarrow \lbrace -1, 1\rbrace$} is used in Corollary \ref{cor:Alex=taut} and Theorem~\ref{thm:main}.

The map $q$ determines a regular cover $M^q$ of $M$ with the deck group isomorphic to~$H$ and the fundamental group isomorphic to the kernel of $q$. This cover admits a veering triangulation $\mathcal{V}^q$ induced by the veering triangulation $\mathcal{V}$ of~$M$. 

The reasoning in the following lemma is completely analogous to the case of factoring the orientation character; see for example \cite[Lemma 1.1]{oriented-cover}. 
\begin{lemma}\label{lem:edge-orient-factors}
	The veering triangulation $\mathcal{V}^q$ is edge-orientable if and only if the edge-orientation homomorphism $\omega: \pi_1(M) \rightarrow \lbrace -1, 1\rbrace$  factors through $q$.
\end{lemma}
\begin{proof}
	The kernel of $q$ is isomorphic to $\pi_1\left(M^q\right)$. To show that $\omega$ factors through $q$ it is enough to show that $\ker q \leq \ker \omega$. This is clearly the case when $\mathcal{V}^q$ is edge-orientable, because then the edge-orientation homomorphism $\omega^q: \pi_1\left(M^q \right) \rightarrow \lbrace -1,1\rbrace$ is trivial. Therefore for every $\gamma \in \ker q \cong \pi_1\left(M^q\right)$ we have $\omega(\gamma)=1$. 
	
	\[\begin{tikzcd}
		\pi_1\left(M^q\right) \arrow[d] \arrow[rd, "\omega^q",shift left=1.5ex] & \\
		\pi_1(M)  \arrow[r, "\omega"] \arrow[d, "q"]& \lbrace -1, 1\rbrace\\
		H  \arrow[ru, dashrightarrow, swap, shift right=1.5ex ]&
	\end{tikzcd}\]
	
	Conversely, suppose that $\mathcal{V}^q$ is not edge-orientable. Then the edge-orientation homomorphism $\omega^q: \pi_1\left(M^q \right) \rightarrow \lbrace -1,1\rbrace$ is not trivial. Let $\gamma \in \ker q \cong \pi_1\left(M^q\right)$ be an element such that $\omega(\gamma)=-1$ and let $\beta \in \pi_1(M)$. Then $q(\beta) = q(\beta\gamma)$ and
	\[\omega(\beta\gamma) = \omega(\beta)\omega(\gamma) = -\omega(\beta) \neq \omega(\beta).\]	
	This implies that $\omega$ does not factor through $q$.
\end{proof}

\section{The taut polynomial and the Alexander polynomial}\label{sec:taut:alex:def}
Let $M$ be a connected 3-manifold equipped with a finite veering triangulation $\mathcal{V}$, with the set~$T$ of tetrahedra, the set $F$ of faces and the set $E$ of edges.  Let
\begin{equation}\label{defn:HM}
	H_M = \bigslant{H_1(M;\zz)}{\text{torsion}}\end{equation}
and $b_1(M) = \mathrm{rank} \ H_M$. 
The regular cover of $M$ determined by $H_M$ is called the \emph{maximal free abelian cover} of~$M$ and denoted by $\Mab$. By $\Vab$ we denote the veering triangulation of $\Mab$ induced by $\mathcal{V}$. 

In this section we recall the definitions of twisted Alexander polynomials $\Delta_M^{\varphi\otimes \pi}$ of $M$ and the taut polynomial $\Theta_{\mathcal{V}}$ of $\mathcal{V}$. We are interested in a particular twisted Alexander polynomial $\Delta_M^{\omega\otimes \pi}$ constructed using the representation $\omega \otimes \pi$ obtained by tensoring the edge-orientation homomorphism $\omega:\pi_1(M) \rightarrow \lbrace -1, 1 \rbrace$ with the projection \mbox{$\pi: \pi_1(M) \rightarrow H_M$} determined by abelianisation and killing the torsion. 
We prove that the taut polynomial $\Theta_\V$ is equal to the twisted Alexander polynomial $\Delta_M^{\omega\otimes \pi}$ and discuss consequences of this result.
Then we prove formulas relating $\Theta_{\mathcal{V}}$ and the untwisted Alexander polynomial~$\Delta_M$ of $M$. There are two formulas, one which holds when the maximal free abelian cover $\Vab$ of~$\mathcal{V}$ is edge-orientable, and the other when it is not.
\subsection{Polynomial invariants of finitely presented $\ZH$-modules}
The definitions of  both the taut and the twisted Alexander polynomials follow the same pattern. Namely, they are derived from the Fitting ideals of certain modules associated to covers of $M$. For that reason in this subsection we recall definitions of Fitting ideals and their invariants.

Let $H$ be a finitely generated free abelian group. Let $\mathcal{M}$ be a finitely presented module over the integral group ring $\zz \lbrack H \rbrack$. Then there exist integers $k, l \in \mathbb{N}$ and an exact sequence
\[\zz\lbrack H \rbrack^k \overset{A}{\longrightarrow} \zz\lbrack H \rbrack^l \longrightarrow \mathcal{M} \longrightarrow  0\]
of $\zz\lbrack H \rbrack$-homomorphisms called a \emph{free presentation} of $\mathcal{M}$. The matrix of $A$, written with respect to any bases of $\zz\lbrack H \rbrack^k$ and $\zz\lbrack H \rbrack^l$, is called a \emph{presentation matrix} for $\mathcal{M}$.
\begin{definition} \cite[Section 3.1]{Northcott}
	Let $\mathcal{M}$ be a finitely presented $\ZH$-module with a presentation matrix $A$ of dimension $l\times k$. We define the $i$-th \emph{Fitting ideal} $\Fit_i(\mathcal{M})$ of~$\mathcal{M}$ to be the ideal in $\ZH$ generated by the $(l-i)\times(l-i)$ minors of $A$. 
\end{definition}
In particular $\Fit_i(\mathcal{M}) = \ZH$ for $i\geq l$, as the determinant of the empty matrix equals 1, and $\Fit_i(\mathcal{M}) = 0$ for $i<0$ or $i<l-k$. The Fitting ideals are independent of the choice of a free presentation of $\mathcal{M}$ \cite[p. 58]{Northcott}. 
\begin{remark*}
	Fitting ideals are called \emph{determinantal ideals} in \cite{Traldi} and \emph{elementary ideals} in \cite[Chapter VIII]{crow_fox}.
\end{remark*}

For any finitely generated ideal $I \subset \ZH$ there exists a unique minimal principal ideal $\overline{I} \subset \ZH$ which contains it \cite[p. 117]{crow_fox}. The ideal $\overline{I}$ is generated by the greatest common divisor of the generators of  $I$ \cite[p. 118]{crow_fox}. This motivates the following definition.
\begin{definition}
Let $\mathcal{M}$ be a finitely presented $\ZH$-module. We define the $i$-th \emph{Fitting invariant} of $\mathcal{M}$ to be the greatest common divisor of elements of $\Fit_i(\mathcal{M})$. When $\Fit_i(\mathcal{M}) = (0)$ we set the $i$-th Fitting invariant of $\mathcal{M}$ to be equal to 0.
\end{definition}
Note that Fitting invariants are well-defined only up to a unit in $\ZH$. In particular, all equalities proved in this paper hold only up to a unit in the appropriate integral group ring. 

\subsection{Twisted Alexander polynomials}\label{subsec:twisted} 

Let $\widetilde{M}$ be the universal cover of $M$. It admits a veering triangulation $\widetilde{\mathcal{V}}$ induced by~$\mathcal{V}$. The sets of ideal tetrahedra, triangles and edges of $\Vuniv$  can be identified with $\pi_1(M)\times T, \pi_1(M)\times F, \pi_1(M)\times E$, respectively. We  orient ideal simplices of $\Vuniv$ in such a way that the restriction of $\widetilde{M}\rightarrow M$ to each ideal simplex is orientation-preserving. 
The free abelian groups generated by $\pi_1(M)\times T$, $\pi_1(M)\times F$ and  $\pi_1(M)\times E$ are isomorphic as $\zz\lbrack \pi_1(M)\rbrack$-modules to the free $\zz\lbrack \pi_1(M)\rbrack$-modules $\zz\lbrack \pi_1(M)\rbrack^T, \zz\lbrack \pi_1(M)\rbrack^F, \zz\lbrack \pi_1(M)\rbrack^E$ generated by $T,F,E$, respectively.

If we truncate ideal vertices of $\widetilde{\V}$, every ideal tetrahedron becomes a polyhedron with four hexagonal faces and four triangular faces.  Hexagonal faces of a truncated ideal tetrahedron correspond to ideal faces of the tetrahedron before truncating, while triangular faces are contained in the boundary $\partial M$. Thus we have the following identifications between the cellular   chain groups of the pair $(\widetilde{M}, \partial \widetilde{M})$ and the aforementioned free  $\zz\lbrack \pi_1(M)\rbrack$-modules generated by $T, F, E$:
\begin{gather*}C_3(\widetilde{M}, \partial \widetilde{M})  \cong \zz\lbrack \pi_1(M)\rbrack^T,\\
C_2(\widetilde{M}, \partial \widetilde{M})  \cong \zz\lbrack \pi_1(M)\rbrack^F,\\
C_1(\widetilde{M}, \partial \widetilde{M})  \cong \zz\lbrack \pi_1(M)\rbrack^E.
\end{gather*}

In other words, the cellullar chain complex $\mathcal{C}(\widetilde{M},\partial \widetilde{M})$ can be identified with
\[0\longrightarrow \zz\lbrack \pi_1(M)\rbrack^T \overset{\widetilde{\partial}_3}{\longrightarrow} \zz\lbrack \pi_1(M)\rbrack^F \overset{\widetilde{\partial}_2}{\longrightarrow} \zz\lbrack \pi_1(M)\rbrack^E \overset{\widetilde{\partial}_1}{\longrightarrow} 0 \longrightarrow 0,\]
where $\widetilde{\partial}_i$ is the cellular boundary map associated to the CW-structure on $(\widetilde{M},\partial \widetilde{M})$ obtained by truncating $\Vuniv$. 

Denote by $\pi: \pi_1(M)\rightarrow H_M$ the natural projection and fix a representation \mbox{$\varphi:\pi_1(M) \rightarrow \mathrm{GL}(k, \zz)$}. These homomorphisms yield  the \emph{tensor representation} 
\begin{gather*}\varphi \otimes \pi: \pi_1(M) \rightarrow \mathrm{GL}(k,\ZHM)\\
	\gamma \mapsto \varphi(\gamma)\cdot \pi(\gamma).
\end{gather*}
Using $\varphi \otimes \pi$ we may view $\ZHM^k$ as a $\zz\lbrack \pi_1(M)\rbrack$-module. Tensoring $\mathcal{C}(\widetilde{M},\partial \widetilde{M})$ with $\ZHM^k$ yields a chain complex $\mathcal{C}^{\varphi\otimes \pi}(M, \partial M)$ of $\ZHM$-modules.
We denote the first homology of $\mathcal{C}^{\varphi\otimes \pi}(M, \partial M)$ by $A^{\varphi \otimes \pi}(M,\partial M)$, and the zeroth Fitting invariant of  $A^{\varphi \otimes \pi}(M,\partial M)$ by $\Delta^{\varphi \otimes \pi}_{(M, \partial M)}$. 
In the case when $\varphi$ is the trivial 1-dimensional representation we use the notation $A(M, \partial M)$ and $\Delta_{(M, \partial M)}$ instead. 
 We can write a free presentation for $A(M, \partial M)$ as
\begin{equation}\label{pres:Alex}\ZHM^F \xrightarrow{\partial} \ZHM^E \longrightarrow A(M, \partial M) \longrightarrow 0,\end{equation}
where $\partial = \partial_2$ is the cellular boundary map associated to the CW-structure on \linebreak $(\Mab, \partial \Mab)$ obtained by truncating $\Vab$. We will always assume that the orientation of a face $f$ is determined by its coorientation and the right hand rule.

By performing the same algebraic operations to the absolute cellular chain complex $\mathcal{C}(\widetilde{M})$ we can define the polynomial $\Delta^{\varphi \otimes \pi}_{M}$, called the \emph{Alexander polynomial} of~$M$ \emph{twisted by the representation} $\varphi \otimes \pi$.  In particular, the untwisted Alexander polynomial~$\Delta_M$ is just the zeroth Fitting invariant of~$H_1(M^{fab};\zz)$. We refer the reader to~\cite{FriedlVidussi_survey} for an exposition on twisted Alexander polynomials.

From Milnor's result on the Reidemeister torsion \cite[Theorem~1']{Milnor_duality} we can deduce that the polynomial $\Delta^{\varphi \otimes \pi}_{(M, \partial M)}$ introduced at the beginning of this subsection is just a twisted Alexander polynomial.

\begin{lemma}[Theorem~1' of \cite{Milnor_duality}]\label{lem:relative=Alex}
Let $M$ be a compact, connected, triangulated 3-manifold
whose boundary is nonempty and consists of tori. Then
\[\pushQED{\qed}
\Delta^{\varphi \otimes \pi}_{(M, \partial M)} = \Delta^{\varphi \otimes \pi}_M.\qedhere
\popQED\]
\end{lemma}
\subsection{The taut polynomial}\label{subsec:taut}
Recall that $M$ denotes a connected 3-manifold equipped with a finite veering triangulation $\mathcal{V}$, and $H_M$ denotes the torsion-free part of $H_1(M;\zz)$. 

The \emph{taut polynomial} $\Theta_{\mathcal{V}}$ of $\mathcal{V}$ is defined in \cite[Section 3]{LMT}  by explicitly giving a free presentation 
\begin{equation}\label{eqn:taut_presentation}\ZHM^F \overset{D}{\longrightarrow} \ZHM^E \longrightarrow \mathcal{E}(\mathcal{V}) \longrightarrow 0\end{equation}
for the \emph{taut module} $\mathcal{E}(\mathcal{V})$. The definition of $D$ depends on the upper track of $\mathcal{V}$, as explained below. 

Let $\tau$ denote the upper track of $\mathcal{V}$. It induces a train track $\tau^{fab}$ in the 2-skeleton of the veering triangulation $\Vab$ of $\Mab$. Let $h\cdot f \in H_M\times F$ be a triangle of $\Vab$. Consider the restriction of $\tau^{fab}$ to $h\cdot f$; see Figure \ref{fig:switch_relation}. This train track determines a \emph{switch relation} between its three half-branches: the large half-branch is equal to the sum of the two small half-branches. By identifying the half-branches with the edges in the boundary of $h\cdot f$ which they meet, we obtain a switch relation between the \emph{edges} in the boundary of $h\cdot f$.  

\begin{figure}[h]
	\begin{center}
		\includegraphics[width=0.23\textwidth]{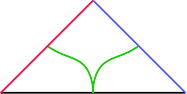} 
		\put(-62,-10){$h_0 \cdot e_0$}
		\put(-109,20){$h_1 \cdot e_1$}
		\put(-13,20){$h_2 \cdot e_2$}
	\end{center}
	\caption{The upper track in a triangle determines a switch relation \\ $h_0\cdot e_0 = h_1\cdot e_1 + h_2\cdot e_2$ between the edges in its boundary.}
	\label{fig:switch_relation}
\end{figure}

We rearrange the switch relation determined by $h\cdot f$ into a linear combination of edges from \mbox{$H_M\times E$}. This linear combination
is the value of $D$ at $h\cdot f$. For example, the image of the triangle presented in Figure \ref{fig:switch_relation} under $D$ is equal to
\[h_0\cdot e_0 - h_1\cdot e_1 - h_2 \cdot e_2 \in \ZHM^E.\]

The \emph{taut polynomial} $\Theta_{\mathcal{V}}$ is defined as the zeroth Fitting polynomial of $\mathcal{E}(\mathcal{V})$. In other words
\[\Theta_{\mathcal{V}} = \gcd \left\lbrace \text{maximal minors of } D \right\rbrace.\]

\subsection{Comparison of polynomials}
Recall from Subsection \ref{subsec:EO_homo} that a veering triangulation $\mathcal{V}$  of $M$ determines the edge-orientation homomorphism  $\omega: \pi_1(M) \rightarrow \lbrace -1, 1\rbrace$. By identifying $\lbrace -1, 1\rbrace$ with $\mathrm{GL}(1, \zz)$ this homomorphism  can be seen as a 1-dimensional representation  $\omega: \pi_1(M) \rightarrow \mathrm{GL}(1, \zz)$. First we show that the taut polynomial of $\mathcal{V}$ is equal to the Alexander polynomial of $M$ twisted by the representation $\omega \otimes \pi$. In the proof we abuse the notation and denote by $\omega$ also the  ring homomorphism
\begin{gather*}\omega: \zz\lbrack \pi_1(M) \rbrack \rightarrow \zz\lbrack \pi_1(M) \rbrack\\
	\omega\left(\sum\limits_{\gamma \in \pi_1(M)} a_\gamma\cdot \gamma \right) = \sum\limits_{\gamma \in \pi_1(M)}\omega(\gamma)\cdot a_\gamma\cdot \gamma.\end{gather*}

\begin{theorem}\label{thm:taut=twisted}
Let $\mathcal{V}$ be a finite veering triangulation of a connected 3-manifold~$M$. 
Then
\[\Theta_{\mathcal{V}} = \Delta_M^{\omega \otimes \pi}\]
where $\omega: \pi_1(M) \rightarrow \lbrace -1, 1\rbrace$ is the edge-orientation homomorphism of $\mathcal{V}$ and \linebreak $\pi:\pi_1(M) \rightarrow H_M$ is the natural projection.
\end{theorem}
\begin{proof}
Observe that the presentation \eqref{pres:Alex} for $A(M, \partial M)$ resembles the presentation \eqref{eqn:taut_presentation} for the $\mathcal{E}(\mathcal{V})$. The only difference is that the presentation matrix $D$ of the taut module does not depend on the orientations of  the triangles and edges of $\Vab$; it only depends on the transverse taut structure on the triangulation.

Given a triangle $f$ of $\mathcal{V}$ we set
\begin{equation}\label{notation:epsilon}\epsilon(f) = \begin{cases}
		1 &\text{if $f$ appears on the left side of its upper large edge}\\
		-1 & \text{otherwise}.
\end{cases}\end{equation}
We work in the universal cover $\widetilde{M}$ of $M$. Let $\widetilde{\mathcal{V}}$ be the veering triangulation of $\widetilde{M}$ induced by $\mathcal{V}$.  Consider the following diagram
\begin{equation}\label{diagram:universal}\begin{tikzcd}
		\zz\lbrack \pi_1(M) \rbrack^F \arrow[r,"\widetilde{\partial}"] \arrow[d, "f \mapsto \epsilon(f)\cdot f"]& \zz\lbrack \pi_1(M) \rbrack^E \arrow[r]\arrow[d, "\omega^{\oplus E}"]& \widetilde{A}(M, \partial M) \arrow[r] \arrow[d]& 0 \\
		\zz\lbrack \pi_1(M) \rbrack^F \arrow[r, "\widetilde{D}"] & \zz\lbrack \pi_1(M) \rbrack^E  \arrow[r]& \widetilde{\mathcal{E}}(\mathcal{V}) \arrow[r]& 0 
\end{tikzcd}\end{equation}
The homomorphisms $\widetilde{\partial}$, $\widetilde{D}$ are defined analogously to $\partial$, $D$, respectively, just in the universal cover of $M$ instead of the maximal free abelian cover of $M$.
The modules $\widetilde{A}(M, \partial M)$, $\widetilde{\mathcal{E}}(\mathcal{V})$ are trivial, but the point here is to compare $\widetilde{\partial}$ with~$\widetilde{D}$.

If $\mathcal{V}$ is edge-orientable choose an orientation on its edges so that it is edge-oriented. We orient ideal simplices of $\widetilde{\mathcal{V}}$ so that the covering map $\widetilde{M} \rightarrow M$ restricted to any ideal simplex is orientation-preserving. Therefore clearly when $\mathcal{V}$ is edge-oriented then for any $\gamma\cdot f \in \pi_1(M) \times F$ we have $\widetilde{\partial}(\gamma\cdot f) = \epsilon(f)\cdot \widetilde{D}(\gamma\cdot f)$.  More generally, the diagram commutes because reversing the orientation of  an edge $\gamma \cdot e$ of $\widetilde{\mathcal{V}}$ if and only if \[\omega(\gamma)=-1\] makes $\widetilde{\mathcal{V}}$ into an edge-oriented triangulation. 

Let us denote by $(\omega \otimes \pi)(\widetilde{\partial})$, $\pi(\widetilde{D})$ the matrices obtained from $\widetilde{\partial}$, $\widetilde{D}$ by mapping its entries through $\omega\otimes \pi$, $\pi$ (extended by linearity to $\zz \lbrack\pi_1(M) \rbrack \rightarrow \ZHM$), respectively.
By definition, $\Delta_{(M, \partial M)}^{\omega \otimes \pi}$ is equal to the greatest common divisor of the maximal minors of $(\omega \otimes \pi)(\widetilde{\partial})$. The above diagram implies that $(\omega \otimes \pi)(\widetilde{\partial})$ differs from $\pi(\widetilde{D}) = D$ only by multiplying some of its columns by -1.  Thus $\Delta_{(M, \partial M)}^{\omega \otimes \pi} = \Theta_{\mathcal{V}}$.
The claim now follows from Lemma~\ref{lem:relative=Alex}.
\end{proof}
Theorem \ref{thm:taut=twisted} can be used to derive various properties of the taut polynomial relying on known results about twisted Alexander polynomials. For instance, we can deduce that it is \emph{symmetric}.  
\begin{definition}
Given $P =\sum\limits_{h \in H_M} a_h \cdot h \in \ZHM$  we define $\overline{P}= \sum\limits_{h \in H_M} a_h \cdot h^{-1}$, and say that $P$ is \emph{symmetric} if there is a unit $u \in \ZHM$ such that $P = u \overline{P}$.
\end{definition}
Symmetric polynomials are also called \emph{palindromic} or \emph{reciprocal}.
Several classical polynomial invariants of 3-manifolds are symmetric, including the Alexander polynomial \cite[Corollary 4.5]{Turaev_Alex} and the Teichm\"uller polynomial \cite[Corollary~4.3]{McMullen_Teich}. Not all twisted Alexander polynomials are symmetric \cite{HillmanSilverWilliams}. However, it is known that if $\varphi \otimes \pi$ satisfies
\begin{equation}\label{eqn:unitary}
(\varphi \otimes \pi)(\gamma) = \overline{\left(((\varphi \otimes \pi)(\gamma))^{-1}\right)^{\mathrm{tr}}}\end{equation}
for any $\gamma \in \pi_1(M)$ 
then $\Delta_M^{\varphi\otimes \pi}$ is symmetric \cite[Proposition 4 in Section 3.3.4]{FriedlVidussi_survey}. 
\begin{corollary}\label{cor:symmetry}
	The taut polynomial of a finite veering triangulation is symmetric. 
\end{corollary}
\begin{proof}
	By Theorem \ref{thm:taut=twisted}, $\Theta_{\mathcal{V}} = \Delta_M^{\omega\otimes \pi}$. Since $\omega\otimes \pi$ is a 1-dimensional representation,  condition \eqref{eqn:unitary} is obviously satisfied. Thus, by \cite[Proposition 4]{FriedlVidussi_survey}, $\Theta_\V$ is symmetric.
	\end{proof}

Another consequence of Theorem \ref{thm:taut=twisted} is that for any 3-manifold $M$ there are only finitely many elements in $\bigslant{\ZHM}{\pm H_M}$ that can be the taut polynomial of some veering triangulation of $M$.
\begin{corollary}\label{cor:finitely:many}
Let $M$ be a 3-manifold. Up to a unit there are only finitely many elements in $\ZHM$ that can be the taut polynomial of a veering triangulation of $M$. In particular, if there exists a 3-manifold with infinitely many veering triangulations then infinitely many of them have the same taut polynomial.
\end{corollary}
\begin{proof}
	By Theorem \ref{thm:taut=twisted}, the taut polynomial of any veering triangulation of $M$ is equal to $\Delta_M^{\omega\otimes \pi}$, where $\omega:\pi_1(M) \rightarrow \lbrace -1, 1\rbrace$. The claim follows because there are only $|H_1(M;\zz/2)|$ different homomorphisms from $\pi_1(M)$ to $\zz/2$.
	\end{proof}

When $\Vab$ is edge-orientable then the edge-orientation homomorphism $\omega$ factors through \mbox{$\sigma: H_M \rightarrow \lbrace -1, 1\rbrace$}. In this case  Theorem \ref{thm:taut=twisted} gives a relation between the taut polynomial of $\mathcal{V}$ and the untwisted Alexander polynomial of~$M$. 

\begin{corollary}\label{cor:Alex=taut}
	Let $\mathcal{V}$ be a veering triangulation of a 3-manifold~$M$. Set $r= b_1(M)$ and let $(h_1, \ldots, h_r)$ be any basis of $H_M$. If $\Vab$ is edge-orientable then
	\[\Theta_{\mathcal{V}}(h_1, \ldots, h_r) = \Delta_M(\sigma(h_1)\cdot h_1, \ldots, \sigma(h_r)\cdot h_r),\]
	where $\sigma: H_M \rightarrow \lbrace -1, 1\rbrace$ is the factor of the edge-orientation homomorphism of $\mathcal{V}$.
\end{corollary}
\begin{proof}
When $\omega$ factors through $\sigma$, the tensor representation $\omega \otimes \pi$ is given by \[\gamma \mapsto \sigma(\pi(\gamma)) \cdot \pi(\gamma).\] Hence the presentation matrix for the twisted module $A^{\omega \otimes \pi}(M, \partial M)$ differs from the presentation matrix for the untwisted module $A(M, \partial M)$ by switching the signs of variables with a nontrivial image under $\sigma$. We obtain that $\Delta_M^{\omega \otimes \pi}$ differs from $\Delta_M$ just by sign changes in variables. The former is equal to $\Theta_{\mathcal{V}}$ by Theorem \ref{thm:taut=twisted}.
\end{proof}

\begin{remark}\label{remark:isomorphism:of:modules}
The proof of Theorem \ref{thm:taut=twisted} yields an isomorphism $A(M, \partial M) \cong \mathcal{E}(\mathcal{V})$  in the case when $\mathcal{V}$ is edge-orientable. If $\mathcal{V}$ is not edge-orientable, then $\sigma: H_M \rightarrow \lbrace -1,1\rbrace$ is nontrivial and  the induced ring homomorphism $\sigma: \ZHM \rightarrow \ZHM$ given by $a_h\cdot h \mapsto \sigma(h)\cdot a_h \cdot h$ is not a $\ZHM$-module homomorphism. For $h\in H_M$ such that  $\sigma(h) = -1$ we have
\[h\cdot\sigma(h) = -h^2,\]
but
\[\sigma(h\cdot h) = h^2.\] 
\end{remark}

Corollary \ref{cor:Alex=taut} allows us to easily identify veering triangulations for which the equality $\Theta_{\mathcal{V}}(h_1, \ldots, h_r) = \Delta_M(\pm h_1,  \ldots, \pm h_r)$ might fail.

\begin{corollary}\label{cor:even_torsion}
	Let $\mathcal{V}$ be a veering triangulation of a 3-manifold $M$. Set $r = b_1(M)$.  If  \[\Theta_{\mathcal{V}}(h_1, \ldots, h_r) \neq \Delta_M(\pm h_1,  \ldots, \pm h_r)\] in $\bigslant{\ZHM}{\pm H_M}$ then the torsion subgroup of $H_1(M;\zz)$  has even order. 
\end{corollary} 
\begin{proof}
	By Corollary \ref{cor:Alex=taut} if 
	\[\Theta_{\mathcal{V}}(h_1, \ldots, h_r) \neq \Delta_M(\pm h_1,  \ldots, \pm h_r)\]
	then $\Vab$ is not edge-orientable. If follows that the deck group of $\Vor \rightarrow \mathcal{V}$ is a quotient of the torsion subgroup of $H_1(M;\zz)$.
\end{proof}

Now suppose that $\Vab$ is not edge-orientable. Denote by $\hVab$ its edge-orientation double cover. It is a (not necessarily maximal) free abelian cover of $\Vor$ with the deck group isomorphic to $H_M$. 

\begin{equation*}\label{diagram:covers}\begin{tikzcd}[row sep = 0.1cm]
	&	& & \Vab \arrow[rrd, "H_M"] &&\\
(\Vor)^{fab} \arrow[r]&\hVab \arrow[rru, "\zz/2"] \arrow[rrd, "H_M", swap]&&& & \mathcal{V}\\
& & & \Vor \arrow[rru, "\zz/2", swap]&&\\
\end{tikzcd}\end{equation*}

Let $\Mor$ denote the manifold underlying $\Vor$. Set $H^{or} = \bigslant{H_1(\Mor;\zz)}{\text{torsion}}$. Since $H^{or}$ surjects onto $H_M$, the ring $\ZHM$ can be regarded as a $\zz\lbrack H^{or} \rbrack$-module. Using this we define the $\ZHM$-module
\[\hat{A}(\V) = A(\Mor, \partial \Mor)\otimes_{\zz\lbrack H^{or} \rbrack} \ZHM.\]

The presentation matrix $\hat{\partial}$ for $\hat{A}(\V)$ can be  obtained from the presentation matrix for $A(\Mor, \partial \Mor)$ by mapping its entries through the epimorphism  $\zz\lbrack H^{or} \rbrack \rightarrow \zz \lbrack  H_M \rbrack$.
As in Section~\ref{sec:EO_cover}, we denote two lifs of an ideal simplex $x$ of $\mathcal{V}$ to $\Vor$ by $x$ and $\overline{x}$. 
Therefore a free presentation for $\hat{A}(\V)$ can be written as
\[\ZHM^F \oplus \ZHM^{\overline F} \xrightarrow{\hat{\partial}} \ZHM^E \oplus \ZHM^{\overline E} \longrightarrow \hat{A}(\V) \longrightarrow 0.\]
 We denote the zeroth Fitting invariant of $\hat{A}(\V)$ by $\hat{\Delta}_{\V}$.
\begin{remark}\label{remark:NEOtwisted}
The polynomial	$\hat{\Delta}_{\V}$ can be interpreted as a twisted Alexander polynomial of $M$. This time we consider the 2-dimensional representation obtained by post-composing the edge-orientation homomorphism $\omega$ of $\mathcal{V}$ with the regular representation of the cyclic group of order 2. That is,
\begin{gather*}
	\rho_\omega: \pi_1(M) \rightarrow \mathrm{GL}(2, \zz)\\
	\gamma\mapsto \begin{cases}
		\begin{bmatrix} 1&0 \\0&1 \end{bmatrix} &\text{if } \omega(\gamma) = 1\\
		\hfill\\
		\begin{bmatrix} 0&1 \\1&0 \end{bmatrix} &\text{if } \omega(\gamma) = -1.\\
	\end{cases}
\end{gather*}
The tensor representation $\rho_\omega \otimes \pi$, where $\pi: \pi_1(M)\rightarrow H_M$,  determines the twisted Alexander polynomial $\Delta_M^{\rho_\omega\oplus \pi} \in \ZHM$. The fact that it is equal to $\hat{\Delta}_{\V}$ follows from \cite[Theorem 3]{FriedlVidussi_survey} and Lemma \ref{lem:relative=Alex}.
\end{remark}
Since $\Vor$ is edge-orientable, Remark \ref{remark:isomorphism:of:modules} implies that the $\zz\lbrack H^{or} \rbrack$-modules $\mathcal{E}(\Vor)$ and $A(\Mor, \partial \Mor)$ are isomorphic. Thus we also have \[\hat{A}(\V) \cong \mathcal{E}(\Vor) \otimes_{\zz\lbrack H^{or} \rbrack} \ZHM.\] Therefore a natural strategy to find a formula relating $\Theta_{\mathcal{V}}$ and $\Delta_M$ in the case when~$\Vab$ is not edge-orientable is to relate them both to  $\hat{\Delta}_\V$. 

Let $X \in \lbrace T,F,E \rbrace$. A chain  $c \in \ZHM^X \oplus \ZHM^{\overline X}$ can be written as
\[c = \sum\limits_{x \in X} \left(	\sum\limits_{h \in H_M} a_h \cdot h \cdot x + \sum\limits_{h \in H_M} b_h \cdot h  \cdot \overline{x}\right),\]
where $a_h$, $b_h \in \zz$. We denote by $\overline{c}$ the chain
\[\overline{c} = \sum\limits_{x \in X}\left(\sum\limits_{h \in H_M} a_h \cdot h\cdot \overline{x} + \sum\limits_{h \in H_M} b_h \cdot h \cdot x\right).\]
Using this notation we find the relation between the images under $\hat{\partial}$ of the two lifts of a triangle $h\cdot f$ of $\Vab$ to $\hVab$.
\begin{lemma}\label{lem:two:lifts:images}
	Let $f$ be a triangle of $\mathcal{V}$. Then
	\[\hat{\partial}(1 \cdot \overline{f}) = - \overline{\hat{\partial}(1\cdot f)}.\]
\end{lemma}
\begin{proof}
	It follows from  Construction~\ref{EOconstruction} that a triangle $1\cdot f$ of $\hVab$ has $h\cdot e$ (respectively $h\cdot \overline{e}$) in its boundary if and only if $1\cdot \overline{f}$ has $h\cdot \overline{e}$ (respectively $h \cdot e$) in its boundary. Moreover, $h \cdot e$ and $h\cdot \overline{e}$ have opposite orientations.
\end{proof}

In order to compare $\partial$ with $\hat{\partial}$ we define a pair of $\ZHM$-module homomorphisms
\begin{align*}
	P_F: & \ \ZHM^F \rightarrow \ZHM^F \oplus \ZHM^{\overline{F}} & P_E:& \ \ZHM^E \rightarrow \ZHM^E \oplus \ZHM^{\overline{E}}\\
	&1\cdot f \mapsto 1\cdot f + 1\cdot \overline{f}& &1\cdot e \mapsto 1\cdot e - 1\cdot \overline{e}.
\end{align*}
\begin{lemma}\label{lem:Alex:up:down}
Let $f$ be a triangle of $\mathcal{V}$. Then
\[P_E \circ \partial (1\cdot f) = \hat{\partial}\circ P_F (1\cdot f).\]
\end{lemma}
\begin{proof}
	By Lemma \ref{lem:two:lifts:images}
	\[\hat{\partial}\circ P_F (1\cdot f) = \hat{\partial}(1\cdot f) - \overline{\hat{\partial}(1\cdot f)}.\]
	
	Let us view $\partial(1\cdot f) \in \ZHM^E$ as an element of $\ZHM^E \oplus \ZHM^{\overline{E}}$. When $\hVab$ is edge-oriented then $\hat{\partial} (1\cdot f)$ differs from $\partial(1\cdot f)$ by adding $\overline{\ \cdot \ }$ to some (possibly none) of the edges and switching the sign of these edges. Hence  $\hat{\partial}(1\cdot f) - \overline{\hat{\partial}(1\cdot f)} = \partial (1\cdot f) - \overline{\partial(1\cdot f)} = P_E \circ \partial (1\cdot f)$.

%
\end{proof}

In order to compare $D$ with $\hat{\partial}$ we define a pair of $\ZHM$-module homomorphisms
\begin{align*}
	Q_F: & \ \ZHM^F \rightarrow \ZHM^F \oplus \ZHM^{\overline{F}} & Q_E:& \ \ZHM^E \rightarrow \ZHM^E \oplus \ZHM^{\overline{E}}\\
	&1\cdot f \mapsto \epsilon(f)(1\cdot f - 1\cdot \overline{f})& &1\cdot e \mapsto 1\cdot e + 1\cdot \overline{e}.
\end{align*}
The meaning of $\epsilon$ is as in \eqref{notation:epsilon}.
\begin{lemma}\label{lem:veer:up:down}
	Let $f$ be a triangle of $\mathcal{V}$. Then
	\[Q_E \circ D (1\cdot f) = \hat{\partial}\circ Q_F (1\cdot f).\]
\end{lemma}
\begin{proof}
	By Lemma \ref{lem:two:lifts:images}
	\[\hat{\partial} \circ Q_F(1\cdot f) = \epsilon(f)\left(\hat{\partial}(1\cdot f) + \overline{\hat{\partial}(1\cdot f)}\right).\]
	Let us view $D(1\cdot f) \in \ZHM^E$ as an element of $\ZHM^E \oplus \ZHM^{\overline{E}}$. When $\hVab$ is edge-oriented then $\epsilon(f) \cdot \hat{\partial}(1\cdot f)$ differs from $D(1\cdot f)$ only by adding $\overline{\ \cdot \ }$ to some (possibly none) of its edges, and not by sign changes. Hence $\epsilon(f)\left(\hat{\partial}(1\cdot f) + \overline{\hat{\partial}(1\cdot f)}\right)$ is equal to
$D(1\cdot f) +\overline{D(1\cdot f)} = Q_E \circ D(1\cdot f)$.
\end{proof}

\begin{proposition}\label{prop:twisted_factors}
Let $\mathcal{V}$ be a finite veering triangulation of a 3-manifold~$M$. If $\Vab$ is not edge-orientable then
\[\hat{\Delta}_{\V} = \Delta_M \cdot \Theta_{\mathcal{V}}.\]
\end{proposition}
\begin{proof}
	By Lemmas \ref{lem:Alex:up:down} and \ref{lem:veer:up:down} there is a commutative diagram
	
	\begin{equation*}\label{diagram:nEO}\begin{tikzcd}
			\ZHM^F\oplus \ZHM^F \arrow[r,"\partial \oplus D"] \arrow[d, "P_F \oplus Q_F"]& \ZHM^E\oplus \ZHM^E \arrow[r]\arrow[d, "P_E \oplus Q_E"]& A(M, \partial M)\oplus \mathcal{E}(\mathcal{V})\arrow[r] \arrow[d]& 0 \\
			\ZHM^F\oplus \ZHM^{\overline{F}} \arrow[r, "\hat{\partial}"] & \ZHM^E \oplus \ZHM^{\overline{E}} \arrow[r]& \hat{A}(\V) \arrow[r]& 0 
	\end{tikzcd}\end{equation*}
In other words, $\hat{A}(\V)$ admits a presentation matrix \[
\begin{bmatrix}\begin{array}{c|c}
	\partial & 0\\
	\hline
	0 & D
	\end{array}
\end{bmatrix}.
\]
	
	Recall from Lemma \ref{lem:relative=Alex} that $\Delta_M = \Delta_{(M, \partial M)}$. Hence if either $\Delta_M = 0$ or $\Theta_{\mathcal{V}} = 0$ then $\hat{\Delta}_{\V} = 0$. Otherwise any nonzero maximal minor of $\hat{\partial}$ is a product of a maximal minor of $D$ and a maximal minor of $\partial$. Let $\lbrace \Theta_i \rbrace_{i\in I}$ be the nonzero maximal minors of~$D$ and let $\lbrace \Delta_j \rbrace_{j \in J}$ be the nonzero maximal minors of $\partial$. 
	Then for each $i \in I$ there is $r_i\in \ZHM$ such that $\Theta_i = \Theta_{\mathcal{V}} \cdot r_i$ and $\gcd(r_i)_{i\in I} = 1$. Similarly, for each $j \in J$ there is $s_j \in \ZHM$ such that
	$\Delta_j = \Delta_M \cdot s_j$ and  $\gcd (s_j)_{j \in J} = 1$. Therefore
	\[\hat{\Delta}_{\V} = \gcd (\Theta_i\cdot \Delta_j)_{i,j} = \Theta_\mathcal{V} \cdot \Delta_M \cdot \gcd(r_is_j)_{i,j} = \Theta_\mathcal{V} \cdot \Delta_M.\qedhere\]
\end{proof}
\section{Dehn filling veering triangulations}\label{sec:Dehn-filling}
Let $M$ be a connected 3-manifold equipped with a finite veering triangulation $\mathcal{V}$. Let~$N$ be a Dehn filling of $M$. We do not assume that $N$ is closed. Following \eqref{defn:HM} we set
\begin{equation*}
	H_N = \bigslant{H_1(N;\zz)}{\text{torsion}}.
\end{equation*}
and  $b_1(N)= \mathrm{rank} \ H_N$. Let $i_\ast: H_M \rightarrow H_N$ be the epimorphism induced by the inclusion of $M$ into $N$. In this section we compare the \emph{specialisation} $i_\ast(\Theta_{\mathcal{V}})$ of the taut polynomial \emph{under the Dehn filling} and the Alexander polynomial $\Delta_N$ of $N$. The motivation behind studying $i_\ast(\Theta_{\mathcal{V}})$ lies in its connection with the Teichm\"uller polynomial \cite[Proposition 7.2]{LMT}. In Section \ref{section:Teich_poly} we apply the main result of this section, Theorem~\ref{thm:main}, to study the Teichm\"uller polynomial and fibred faces of the Thurston norm ball.

%

\subsection{The specialisation $i_\ast(\Delta_M) \in \ZHN$}
The first step in finding a relation between $i_\ast(\Theta_{\mathcal{V}})$ and $\Delta_N$ is to compare $i_\ast(\Delta_M)$ with $\Delta_N$.
This has been previously studied by Turaev. By first recursively applying the formulas given in Corollary 4.2 of \cite{Turaev_Alex} and then using Corollary~4.1 of \cite{Turaev_Alex} we obtain the following lemma.
\begin{lemma}\label{lem:i_Alex_under_filling}
	Let $M$ be a compact, orientable, connected 3-manifold with nonempty boundary consisting of tori. Let $N$ be any Dehn filling of $M$ such that $b_1(N)$ is positive. Denote by $ \ell_1, \ldots, \ell_k$ the core curves of the filling solid tori  in $N$ and by \mbox{$i_\ast: H_M \rightarrow H_N$} the epimorphism induced by the inclusion of $M$ into $N$. Assume that for every \mbox{$j \in \lbrace 1, \ldots, k \rbrace$} the class $\lbrack \ell_j\rbrack \in H_N$ is nontrivial.
	\begin{enumerate}[leftmargin=0.5cm, label={{\Roman*}}. ]
		\item $b_1(M) \geq 2$ and
		\begin{enumerate}[leftmargin=0.5cm]
			\item $b_1(N) \geq 2$. Then
			\[i_\ast(\Delta_M) = \Delta_N \cdot \prod\limits_{j=1}^k (\lbrack \ell_j \rbrack - 1).\]
			\item $b_1(N) =1$. Let $h$ denote the generator of $H_N$.
			\begin{itemize}[leftmargin=0.5cm]
				\item If $\partial N \neq \emptyset$ then 
				\[i_\ast(\Delta_M) = (h-1)^{-1}\cdot \Delta_N \cdot \prod\limits_{j=1}^k (\lbrack \ell_j \rbrack - 1).\]
				\item If $N$ is closed 
				\[i_\ast(\Delta_M) = (h-1)^{-2}\cdot \Delta_N \cdot \prod\limits_{j=1}^k (\lbrack \ell_j \rbrack - 1).\]
			\end{itemize}
		\end{enumerate}
		\item $b_1(M) = 1$. Let $h$ denote the generator of $H_N \cong H_M$.
		\begin{enumerate}
			\item If $\partial N \neq \emptyset$  then
			\[i_\ast(\Delta_M) = \Delta_N.\]
			\item If $N$ is closed set $\ell = \ell_1$ and then
			\[\pushQED{\qed}
			i_\ast(\Delta_M) = (h-1)^{-1}(\lbrack \ell \rbrack - 1) \cdot \Delta_N.\qedhere
			\popQED\]
		\end{enumerate}
	\end{enumerate}
\end{lemma}
\subsection{The specialisation $i_\ast(\Theta_{\mathcal{V}}) \in \ZHN$}
Let $N^{fab}$ denote the maximal free abelian cover of~$N$. By removing the preimages of the filling solid tori under $N^{fab} \rightarrow N$ we obtain a free abelian cover $M^N$ of $M$. The deck group of the covering $M^N \rightarrow M$ is isomorphic to $H_N$. This situation is illustrated in the following commutative diagram.
 \begin{equation}\label{diagram:MN_Nab}\begin{tikzcd}
 		\Mab \arrow[r] & M^N \arrow[r] \arrow[d]& M\arrow[d,"i"]\\
 		&N^{fab} \arrow[r]& N.
 \end{tikzcd}\end{equation}
The veering triangulation $\mathcal{V}$ of $M$ induces a veering triangulation $\mathcal{V}^N$ of $M^N$. By Lemma \ref{lem:edge-orient-factors} if $\mathcal{V}^N$ is edge-orientable, then the edge-orientation homomorphism of $\mathcal{V}$ factors through $H_N$. We denote this factor by $\sigma_N: H_N \rightarrow  \lbrace -1, 1 \rbrace$.
 
 \begin{theorem}\label{thm:main}
 	Let $\mathcal{V}$ be a finite veering triangulation of a connected 3-manifold~$M$. Let~$N$ be a Dehn filling of~$M$ such that $s=b_1(N)$ is positive. Denote by $\ell_1, \ldots, \ell_k$ the core curves of the filling solid tori in $N$, and by $i_\ast: H_M \rightarrow H_N$ the epimorphism induced by the inclusion of $M$ into $N$.	
 	Assume that the veering triangulation $\mathcal{V}^N$ is edge-orientable and that for every $j \in \lbrace 1, \ldots, k \rbrace$ the class $\lbrack \ell_j \rbrack \in H_N$ is nontrivial. Let \mbox{$\sigma_N: H_N \rightarrow \lbrace -1, 1 \rbrace$} be the homomorphism through which the edge-orientation homomorphism of $\mathcal{V}$  factors.
 	\begin{enumerate}[leftmargin=0.5cm, label={{\Roman*}}. ]
 		\item $b_1(M) \geq 2$ and
 		\begin{enumerate}[leftmargin=0.5cm]
 			\item $s \geq 2$. Then
 			\[i_\ast(\Theta_\mathcal{V})(h_1, \ldots, h_s) = \Delta_N (\sigma_N(h_1)\cdot h_1, \ldots, \sigma_N(h_s)\cdot h_s)\cdot \prod\limits_{j=1}^k (\lbrack \ell_j \rbrack - \sigma_N(\lbrack \ell_j \rbrack)).\]
 			\item $s  =1$. Let $h$ denote the generator of $H_N$.
 			\begin{itemize}[leftmargin=0.5cm]
 				\item If $\partial N \neq \emptyset$ then 
 				\[i_\ast(\Theta_\mathcal{V})(h) = (h-\sigma_N(h))^{-1}\cdot \Delta_N (\sigma_N(h)\cdot h) \prod\limits_{j=1}^k (\lbrack \ell_j \rbrack - \sigma_N(\lbrack \ell_j \rbrack)).\]
 				\item If $N$ is closed then 
 				\[i_\ast(\Theta_\mathcal{V})(h) = (h-\sigma_N(h))^{-2}\cdot \Delta_N(\sigma_N(h)\cdot h) \cdot \prod\limits_{j=1}^k (\lbrack \ell_j \rbrack - \sigma_N(\lbrack \ell_j \rbrack)).\]
 			\end{itemize}
 		\end{enumerate}
 		\item $b_1(M) = 1$. Let $h$ denote the generator of $H_N \cong H_M$.
 		\begin{enumerate}
 			\item If $\partial N \neq \emptyset$  then
 			\[i_\ast(\Theta_\mathcal{V})(h) = \Delta_N(\sigma_N(h)\cdot h).\]
 			\item If $N$ is closed set $\ell = \ell_1$, and then
 			\[
 			i_\ast(\Theta_\mathcal{V})(h) = (h-\sigma_N(h))^{-1}(\lbrack \ell \rbrack - \sigma_N(\lbrack \ell \rbrack)) \cdot \Delta_N(\sigma_N(h)\cdot h).\]
 		\end{enumerate}
 	\end{enumerate}
 	
 	In particular  with the above assumptions we have
 	\[i_\ast(\Theta_{\mathcal{V}})(h_1, h_2, \ldots, h_s) = \Delta_N(\pm h_1,  \ldots, \pm h_s)\]
 	if and only if one of the following four conditions holds
 	\begin{itemize}
 		\item $N=M$, or
 		\item  $b_1(N)=1$, $\partial N \neq \emptyset$, $k=1$, and $\lbrack \ell_1 \rbrack$ generates $H_N$, or
 		\item $b_1(N)=1$, $N$ is closed, $k=2$ and $\lbrack \ell_1 \rbrack = \lbrack \ell_2 \rbrack$ generates $H_N$, or
 		\item $b_1(M)=1$, $N$ is closed and $\lbrack \ell_1 \rbrack$ generates $H_N$.
 	\end{itemize}
 \end{theorem}
 \begin{proof}
 	Set $r= b_1(M)$. First note that if $\mathcal{V}^N$ is edge-orientable, then $\Vab$ is edge-orientable as well. Therefore there is a homomorphism $\sigma: H_M \rightarrow \lbrace -1,1 \rbrace$ through which the edge-orientation homomorphism of $\mathcal{V}$ factors. By Corollary \ref{cor:Alex=taut} if $\Vab$ is edge-orientable, then
 	\[\Theta_{\mathcal{V}}(g_1, \ldots, g_r) = \Delta_M(\sigma(g_1)\cdot g_1, \ldots, \sigma(g_r)\cdot g_r).\] 
 	
 	Since $\mathcal{V}^N$ is edge-orientable, the kernel of $i_\ast$ is contained in the kernel of $\sigma$ and $\sigma(g) = \sigma_N(i_\ast(g))$ for every $g \in H_M$.
 	Therefore
 	\[i_\ast(\Theta_{\mathcal{V}})(h_1, \ldots, h_s) = i_\ast(\Delta_M)(\sigma_N(h_1)\cdot h_1, \ldots, \sigma_N(h_s)\cdot h_s),\]
 	where $s = b_1(N)$.
 	Now the claim follows from Lemma \ref{lem:i_Alex_under_filling}.
 \end{proof}
 
 \subsection{Edge-orientability of $\mathcal{V}^N$}

We already noted in the proof of Corollary \ref{cor:even_torsion} that the veering triangulation $\Vab$ of the \emph{maximal} free abelian cover of $M$ can fail to be edge-orientable only if the torsion subgroup of $H_1(M;\zz)$ is of even order. 	
Edge-orientability of $\mathcal{V}^N$ additionally depends on the Dehn filling slopes that produce $N$ from~$M$.

Suppose that the truncated model of~$M$ has $b$ boundary components $T_1, \ldots, T_b$. Let~$\gamma_j$ be a Dehn filling slope on $T_j$. 
Suppose that $N$ is obtained from $M$ by Dehn filling $T_j$ along $\gamma_j$. Then $\mathcal{V}^N$ is edge-orientable if and only if $\Vab$ is edge-orientable and moreover $\sigma(\gamma_j)=1$ for every $j \in \lbrace 1,\ldots, b \rbrace$. (Here we identify~$\gamma_j$ with the image of its homology class in $H_1(T_j;\zz)$ under $H_1(T_j;\zz) \rightarrow H_M$.) 

Suppose that $\Vab$ is edge-orientable. 
If the preimage of every torus boundary component of~$M$ under the covering map $\Vor \rightarrow \mathcal{V}$ is disconnected then any Dehn filling~$N$ of~$M$ determines  an edge-orientable veering triangulation~$\mathcal{V}^N$.
More generally, denote by~$\mathcal{I}$ the subset of $\lbrace 1, \ldots, b \rbrace$ indicating the boundary tori of $M$ with a connected preimage under $\Vor \rightarrow \mathcal{V}$. For $j \in \mathcal{I}$ let $(a_j, b_j)$ be a basis of $H_1(T_j;\zz)$ such that $\sigma(a_j) = -1$ and $\sigma(b_j)=1$. Then $\gamma_j = x_ja_j + y_jb_j$ for some $x_j, y_j \in \zz$ such that $\gcd(x_j, y_j)=1$. The triangulation $\mathcal{V}^N$ is edge-orientable if and only if $x_j \in 2\zz$ for every  $j \in \mathcal{I}$.

\begin{remark*}There are 87047 veering triangulations in the Veering Census \cite{VeeringCensus}. 62536 (71.8\%) of them are not edge-orientable.	
	Out of 62536 not edge-orientable veering triangulations there are 49637 (79.4\%) whose edge-orientable double cover $\Vor$ has the same number of cusps as $\mathcal{V}$. There are only 5854 (9.4\%) whose edge-orientable double cover $\Vor$ has twice as many cusps as $\mathcal{V}$.
\end{remark*}

\section{Consequences for the Teichm\"uller polynomial and fibred faces}\label{section:Teich_poly}
The taut polynomial of a veering triangulation is related to an older polynomial invariant of 3-manifolds called the \emph{Teichm\"uller polynomial}. It was introduced by McMullen in \cite{McMullen_Teich} and is associated to a \emph{fibred face of the Thurston norm ball}.
In this section we interpret Theorem \ref{thm:main} as a result relating the Alexander polynomial of a 3-manifold and the Teichm\"uller polynomial of its fibred face. We interpret this relation in terms of orientability of invariant laminations in the fibres.
\subsection{Fibred faces of the Thurston norm ball and the Teichm\"uller polynomial}
If $N$ is a compact, oriented, hyperbolic 3-manifold then $H^1(N;\rr) \cong H_2(N,\partial N;\rr)$ is equipped with the \emph{Thurston norm} $\Thnorm{\cdot}$ \cite{Thur_norm}. 
The unit norm ball $B_{\mathrm{Th}}(N)$ with respect to $\Thnorm{\cdot}$ is a finite-sided polytope with vertices in $H^1(N; \mathbb{Q})$ \cite[Theorem 2]{Thur_norm}. Thurston observed that if~$S$ is the fibre of a fibration of $N$ over the circle then the homology class~$\lbrack S \rbrack$ lies in the interior of the cone $\rr_+\hspace{-0.1cm} \cdot \face{F}$ on some top-dimensional face $\face{F}$ of $B_{\mathrm{Th}}(N)$. Moreover, in this case every  integral primitive class $\lbrack S' \rbrack$ from the interior of $\rr_+\hspace{-0.1cm} \cdot \face{F}$ determines a fibration of~$N$ over the circle \cite[Theorem 3]{Thur_norm}. 
Top-dimensional faces of $B_{\mathrm{Th}}(N)$ with the above property are called \emph{fibred faces} of the Thurston norm ball in $H^1(N;\rr)$.

Let $N$ be as above. Let $\face{F}$ be a fibred face of the Thurston norm ball in $H^1(N;\rr)$. Picking an integral primitive class $\xi$ from the interior of the cone $\rr_+\hspace{-0.1cm}\cdot \face{F}$ fixes an expression of $N$ as the mapping torus
\[N = N_\psi =\bigslant{\left(S \times \lbrack 0, 1 \rbrack\right)}{\lbrace (x,1)\sim (\psi(x),0) \rbrace}\]
of a pseudo-Anosov homeomorphism $\psi: S\rightarrow S$ \cite[Proposition 2.6]{Thurston_map_tor}. The surface~$S$ is the \emph{taut representative} of $\xi$, which means that it satisfies  $\chi(S) = -\Thnorm{\xi}$. The homeomorphism $\psi$ is called the \emph{monodromy} of the fibration $S \rightarrow N \rightarrow S^1$.
It determines a pair of 1-dimensional laminations in $S$ which are invariant under $\psi$. Let $\lambda$ denote the \emph{stable} lamination of~$\psi$. The mapping torus~$\mathcal{L}$ of~$\lambda$ gives the stable lamination of the suspension flow on $N$ determined by $\xi$.  

McMullen proved that up to isotopy the lamination $\mathcal{L}$ does not depend on the chosen primitive integral class from the interior of $\rr_+\hspace{-0.1cm}\cdot \face{F}$ \cite[Corollary 3.2]{McMullen_Teich}. This also follows from a result of Fried \cite[Theorem~14.11]{Fried_suspension} which says that fibrations lying over the same fibred face determine isotopic suspension flows.
Using this McMullen defined the \emph{Teichm\"uller polynomial}~$\Theta_{\texttt{F}}$ of $\face{F}$ as  the zeroth Fitting invariant of the \emph{module of transversals} to the lamination induced by $\mathcal{L}$ in the maximal free abelian cover of $N$ \cite[Section~3]{McMullen_Teich}. The main feature of this polynomial is that it encodes information on the stretch factors of monodromies of all fibrations lying in the fibred cone $\rr_+\hspace{-0.1cm}\cdot \face{F}$. More precisely, by viewing $\xi$ as a homomorphism $\xi: H_N \rightarrow \zz$ we can consider the 1-variable polynomial $\xi(\Theta_{\mathtt{F}})$, the \emph{specialisation} of $\Theta_{\mathtt{F}}$ at $\xi$.
McMullen showed that the largest real root of $\xi(\Theta_{\mathtt{F}})$ is equal to the stretch factor of $\psi$ \cite[Theorem~5.1]{McMullen_Teich}.

\subsection{The Teichm\"uller polynomial and the taut polynomial}
Let $N$ be a compact, oriented, hyperbolic 3-manifold which is fibred over the circle. Let $\face{F}$ be a fibred face of the Thurston norm ball in $H^1(N;\rr)$. It determines a unique (up to isotopy and reparametrisation) suspension flow $\Psi$ on $N$ \cite[Theorem~14.11]{Fried_suspension}. By $\mathrm{sing}(\face{F})$ we denote the set $\lbrace \ell_1, \ldots, \ell_k\rbrace$ of the  \emph{singular orbits} of $\Psi$. If $\mathrm{sing}(\face{F}) = \emptyset$ we say that $\face{F}$ is \emph{fully-punctured}.

If we pick a fibration lying over $\face{F}$ we can follow Agol's algorithm \cite[Section~4]{Agol_veer} to construct a \emph{layered} veering triangulation $\mathcal{V}$ of $M =  N \bez \mathrm{sing}(\face{F})$. This triangulation does not depend on the chosen fibration from $\rr_+\hspace{-0.1cm}\cdot \face{F}$ \cite[Proposition 2.7]{MinskyTaylor}. 
Landry, Minsky and Taylor observed that we can compute the Teichm\"uller polynomial of $\face{F}$ using the taut polynomial of $\mathcal{V}$.
\begin{lemma}[Proposition 7.2 of \cite{LMT}] \label{lem:computing_Teich}
	Let $N$ be a compact, connected, oriented, hyperbolic 3-manifold. Let  $\face{F}$ be a fibred face of the Thurston norm ball in $H^1(N;\rr)$. Denote by~$\mathcal{V}$ the veering triangulation of  $M =  N \bez \mathrm{sing}(\face{F})$ associated to $\face{F}$. Let $i_\ast: H_M \rightarrow H_N$ be the epimorphism induced by the inclusion of $M$ into $N$. Then
	\[\pushQED{\qed} 
	\Theta_{\mathtt{F}} = i_\ast(\Theta_{\mathcal{V}}). \qedhere
	\popQED\] 	
\end{lemma}
An immediate corollary of Lemma \ref{lem:computing_Teich} and Theorem \ref{thm:taut=twisted} is that the Teichm\"uller polynomials are just specialisations of twisted Alexander polynomials.
\begin{corollary}\label{cor:Teich:twisted}
Let  $\face{F}$ be a fibred face of the Thurston norm ball in $H^1(N;\rr)$. Let $M = N \bez\mathrm{sing}(\face{F})$. Then the Teichm\"uller polynomial $\Theta_{\mathtt{F}}$ of $\face{F}$ is a specialisation of a twisted Alexander polynomial of $M$. \qed
\end{corollary}

	An algorithm to compute the Teichm\"uller polynomial using Lemma \ref{lem:computing_Teich} is given in \cite[Section 8]{taut_veer_teich}. As mentioned in the Introduction, Corollary \ref{cor:Teich:twisted} implies a much faster algorithm to compute the Teichm\"uller polynomial via Fox calculus.
	
	 Furthermore, Lemma \ref{lem:computing_Teich} allows us to interpret Theorem \ref{thm:main} as a theorem which relates the Teichm\"uller polynomial with the Alexander polynomial.  The details of that relationship are given in Corollary \ref{cor:Teich_in_one}. As an application, we prove equality
	between a fibred face $\face{F}$ of the Thurston norm ball and a face $\face{A}$ of the Alexander norm ball under a weaker condition on $\face{F}$ than McMullen's condition from \cite[Theorem 7.1]{McMullen_Teich}. In Subsection \ref{subsec:orientability} we explain how orientable fibred classes control the formulas given in Corollary \ref{cor:Teich_in_one}.
\subsection{The Teichm\"uller polynomial and the Alexander polynomial}
In this subsection  we prove lemmas that allow us to interpret Theorem \ref{thm:main} as a strenghtening of Theorem~7.1 of \cite{McMullen_Teich}. This is summarised in Corollary \ref{cor:Teich_in_one}.

Recall that given a fibred face $\face{F}$ of the Thurston norm ball in $H^1(N;\rr)$ there is
\begin{itemize}
	\item a unique (up to isotopy) 2-dimensional lamination $\mathcal{L}$ in $N$ associated to $\face{F}$ \cite[Corollary 3.2]{McMullen_Teich},
	\item a unique veering triangulation $\mathcal{V}$ of $M=N \bez \mathrm{sing}(\face{F})$ carrying fibres from $\rr_+\hspace{-0.1cm}\cdot \face{F}$ punctured at the singularities of the monodromies \cite[Proposition 2.7]{MinskyTaylor}.
\end{itemize}
The lamination $\mathcal{L}$ misses $\mathrm{sing}(\face{F})$ and therefore can also be seen as a lamination in $M$.
 Consistently with the diagram \eqref{diagram:MN_Nab} we denote the laminations in $M^N$, $M^{fab}$ induced by $\mathcal{L} \subset M$ by $\mathcal{L}^N$, $\mathcal{L}^{fab}$, respectively. The lamination $\mathcal{L}^N$ can also be seen inside $N^{fab}$.
 
By Lemma \ref{lemma:EO:to}, $\mathcal{L}$ is transversely orientable if and only if $\mathcal{V}$ is edge-orientable. This property passes to covers of $M$.
\begin{corollary} \label{cor:trans_or_edge_or}
Let $\face{F}$ be a fibred face of the Thurston norm ball in $H^1(N;\rr)$. Let
	\begin{itemize}
		\item $\mathcal{L}$ be the 2-dimensional lamination in $N$ associated to $\face{F}$,
		\item $\mathcal{V}$ be the veering triangulation of $M=N \bez \mathrm{sing}(\face{F})$ associated to $\face{F}$.
	\end{itemize}  Then the lamination $\mathcal{L}^N$ in $N^{fab}$ is transversely orientable if and only if the veering triangulation $\mathcal{V}^N$ of $M^N$ is edge-orientable. \qed
\end{corollary}

Recall that by $\mathrm{sing}(\face{F}) = \lbrace \ell_1, \ldots, \ell_k \rbrace$ we denote the singular orbits of the flow canonically associated to $\face{F}$. In the language of Theorem \ref{thm:main} they correspond to the core curves of the filling solid tori in $N$. In Theorem \ref{thm:main} we needed to assume that their classes in $H_N = \bigslant{H_1(N;\zz)}{\text{torsion}}$ are nontrivial. In the fibred setting we know that the classes $\lbrack \ell_j \rbrack$   have  nonzero algebraic intersection with every $\lbrack S \rbrack \in \mathrm{int}(\rr_+\hspace{-0.1cm}\cdot \face{F}) \cap H_2(N, \partial N;\zz)$. Therefore we get the following lemma.
\begin{lemma}\label{lem:sing_orb_nontrivial}
	Let $\face{F}$ be a fibred face of the Thurston norm ball in $H^1(N;\rr)$. Let $\mathrm{sing}(\face{F}) = \lbrace \ell_1, \ldots, \ell_k \rbrace$ be the singular orbits of the flow associated to $\face{F}$. Then for every \linebreak $j \in \lbrace 1, 2, \ldots, k\rbrace$ the class $\lbrack \ell_j \rbrack$ in $H_N$ is nontrivial. \qed
\end{lemma}

Using Lemma \ref{lem:computing_Teich}, Corollary \ref{cor:trans_or_edge_or}, Lemma \ref{lem:sing_orb_nontrivial} and Theorem \ref{thm:main} we derive relations between the Teichm\"uller polynomial of $\face{F}$ and the Alexander polynomial of~$N$. 
\begin{corollary}\label{cor:Teich_in_one}
	Let $N$ be a compact, connected, oriented, hyperbolic 3-manifold with a fibred face $\face{F} \subset H^1(N;\rr)$ of the Thurston norm ball. Let $\mathrm{sing}(\face{F}) = \lbrace \ell_1, \ldots, \ell_k \rbrace$. Denote by~$\mathcal{L}$ the stable lamination of the suspension flow associated to $\face{F}$. Assume that the lamination  in~$N^{fab}$ induced by $\mathcal{L}$ is  transversely orientable.	
	Set  $s= b_1(N)$ and $M = N \bez \mathrm{sing}(\face{F})$.
	\begin{enumerate}[leftmargin=0.5cm, label={{\Roman*}}. ]
		\item $b_1(M) \geq 2$ and
		\begin{enumerate}[leftmargin=0.5cm]
			\item $s\geq 2$. Then
			\[\Theta_{\mathtt{F}}(h_1, \ldots, h_s) = \Delta_N (\sigma_N(h_1)\cdot h_1, \ldots, \sigma_N(h_s)\cdot h_s)\cdot \prod\limits_{j=1}^k (\lbrack \ell_j \rbrack - \sigma_N(\lbrack \ell_j \rbrack)).\]
			\item $s =1$. Let $h$ denote the generator of $H_N$.
			\begin{itemize}[leftmargin=0.5cm]
				\item If $\partial N \neq \emptyset$ then 
				\[\Theta_{\mathtt{F}}(h) = (h-\sigma_N( h ))^{-1}\cdot \Delta_N (\sigma_N(h)\cdot h) \prod\limits_{j=1}^k (\lbrack \ell_j \rbrack - \sigma_N(\lbrack \ell_j \rbrack)).\]
				\item If $N$ is closed then 
				\[\Theta_{\mathtt{F}}(h) = (h-\sigma_N( h ))^{-2}\cdot \Delta_N(\sigma_N(h)\cdot h) \cdot \prod\limits_{j=1}^k (\lbrack \ell_j \rbrack - \sigma_N(\lbrack \ell_j \rbrack)).\]
			\end{itemize}
		\end{enumerate}
		\item $b_1(M) = 1$. Let $h$ denote the generator of $H_N \cong H_M$.
		\begin{enumerate}
			\item If $\partial N \neq \emptyset$  then
			\[\Theta_{\mathtt{F}}(h) = \Delta_N(\sigma_N(h)\cdot h).\]
			\item If $N$ is closed set $\ell = \ell_1$, and then
			\[\pushQED{\qed}
			\Theta_{\mathtt{F}}(h) = (h-\sigma_N( h ))^{-1}(\lbrack \ell \rbrack - \sigma_N(\lbrack \ell \rbrack)) \cdot \Delta_N(\sigma_N(h)\cdot h).\qedhere
			\popQED\]
		\end{enumerate}
	\end{enumerate}
\end{corollary}
 
In \cite{McMullen_Alex} McMullen used the Alexander polynomial to define the \emph{Alexander norm} on $H^1(N;\rr)$. Suppose that $\Delta_N = \sum\limits_{h \in H_N} a_h \cdot h$. Then the Alexander norm of $\xi \in H^1(N;\rr)$ is given by 
\[\Anorm{\xi} = \sup \lbrace \xi(h) - \xi(g) \ | \ a_ha_g \neq 0\rbrace.\] 
The unit norm ball of the Alexander norm is equal to the polar dual of  the Newton polytope of $\Delta_N$ scaled down by a factor of 2.
McMullen proved that when \mbox{$b_1(N)>1$} the Alexander norm is a lower bound for the Thurston norm \cite[Theorem~1.1]{McMullen_Alex}. Equality holds on cohomology classes corresponding to fibrations of~$N$  over the circle. Thus every fibred face of the Thurston norm ball is contained in a face of the Alexander norm ball. Dunfield showed that the link complement L10n14 has a fibred face which is only properly contained in a face of the Alexander norm ball \cite[Section~6]{Dunfield_Alex}. It was then proved by McMullen that a fibred face is equal to a face of the Alexander norm ball if the stable lamination of the flow associated to the face is transversely orientable \cite[Theorem~7.1]{McMullen_Teich}. Corollary \ref{cor:Teich_in_one} implies that in fact it is enough to assume transverse orientability of  the preimage of that lamination in the maximal free abelian cover of~$N$.
\begin{corollary}\label{cor:faces:equal}
	Let $N$ be a compact, oriented, hyperbolic 3-manifold with a fibred face $\face{F}$ of the Thurston norm ball in $H^1(N;\rr)$. Let $\mathcal{L}$ be  the stable lamination of the suspension flow associated to $\face{F}$. If the lamination induced by $\mathcal{L}$ in the maximal free abelian cover of $N$ is transversely orientable then there is a face $\face{A}$ of the Alexander norm ball such that $\face{A} = \face{F}$. 
\end{corollary}
\begin{proof}
	It follows from Corollary \ref{cor:Teich_in_one} that the stretch factors of monodromies of fibrations lying over~$\face{F}$ can be computed by specialising 
	$\Delta_N (\sigma_N(h_1)\cdot h_1, \ldots, \sigma_N(h_s)\cdot h_s)$. Since the Newton polytopes of $\Delta_N$ and $\Delta_N (\sigma_N(h_1)\cdot h_1, \ldots, \sigma_N(h_s)\cdot h_s)$ are equal, Theorem A.1(C) of \cite{McMullen_Teich} implies that $\rr_+\hspace{-0.1cm}\cdot \face{F}$ is equal to the cone on some face $\face{A}$ of the Alexander norm ball in $H^1(N;\rr)$. Since by \cite[Theorem~1.1]{McMullen_Alex} we also have $\face{F} \subset \face{A}$, we must have $\face{F} = \face{A}$.
\end{proof}

\begin{remark}
	A fibred face of the Thurston norm ball can be equal to a face of the Alexander norm ball even if the associated lamination in the maximal free abelian cover is not transversely orientable. For instance, let $\mathcal{V}$ be the veering triangulation which in the Veering Census \cite{VeeringCensus} is described by the string \begin{center}\texttt{oLLLLLPwQQcccefgijlmkklnnnlnewbnetafobnkj\_12001112122200}.\end{center} 
	
	Then $\mathcal{V}$ is layered, $\Vab$ is not edge-orientable and
	 \begin{align*}
			\Theta_\mathcal{V} &=  a^7b - a^6b^2 - a^5b^3 + a^4b^4 - a^6b - 2a^5b^2 + 2a^4b^3 + 2a^3b^4 - ab^6 +\\&- a^6 + 2a^4b^2 + 2a^3b^3 - 2a^2b^4 - ab^5 + a^3b^2 - a^2b^3 - ab^4 + b^5\\
			\Delta_M &=	a^7b + a^6b^2 + a^5b^3 + a^4b^4 + a^6b + 2a^4b^3 + 2a^3b^4 + 2a^2b^5 + ab^6 +\\&+ a^6 + 2a^5b + 2a^4b^2 + 2a^3b^3 + ab^5 + a^3b^2 + a^2b^3 + ab^4 + b^5.
		\end{align*}
One can check that $\Theta_{\mathcal{V}}$ and $\Delta_M$ have the same Newton polytopes. Theorem 6.1 of \cite{McMullen_Teich} then implies that the fibred face determined by $\mathcal{V}$ is equal to a face of the Alexander norm ball. 
\end{remark}
\subsection{Orientability of invariant laminations in the fibre}\label{subsec:orientability}
One of the consequences of Corollary \ref{cor:Teich_in_one} is that the Teichm\"uller polynomials of distinct fibred faces of the same manifold are often almost the same.
\begin{corollary}\label{cor:teich_polys_of_faces}
	Let $N$ be a compact, oriented, hyperbolic 3-manifold which is fibred over the circle. Let $\face{F}_1$, $\face{F}_2$ be two fibred faces of the Thurston norm ball in $H^1(N;\rr)$. For $i=1,2$ denote by~$\mathcal{L}_i$ the lamination associated to $\face{F}_i$ and by $\Theta_i$ the Teichm\"uller polynomial of $\face{F}_i$. If the induced laminations $\mathcal{L}_1^N, \mathcal{L}_2^N$ in $N^{fab}$ are transversely orientable then 
	\[\Theta_1(h_1, \ldots, h_s) = P(h_1, \ldots, h_s) \cdot \Theta_2(\pm h_1, \ldots, \pm h_s)\]
	where $P$ is a product of factors of the form $(h\pm1)$ and $(h\pm1)^{-1}$,  $h \in H_N$. \qed
\end{corollary}

This corollary is not surprising. 
Let $\xi \in H^1(N;\zz)$ be a primitive integral cohomology class from the interior of the cone over a fibred face~$\face{F}$. Let $\psi$ denote the monodromy of the corresponding fibration. By \cite[Theorem 5.1]{McMullen_Teich}, the largest real root of $\xi(\Theta_{\texttt{F}})$ is equal to the stretch factor of $\psi$.
On the other hand, the largest in the absolute value real root of $\xi(\Delta_N)$ is equal to the homological stretch factor of~$\psi$ \cite[Assertion~4]{Milnor_covers}. There is an easy criterion which tells when these two numerical invariants of $\psi$ are equal.
\begin{theorem}[Lemma 4.3 of \cite{BandBoyland}] \label{thm:orientable:lam}
	Let $\psi: S\rightarrow S$ be a pseudo-Anosov homeomorphism of an orientable surface. The stretch factor and the homological stretch factor of $\psi$ are equal if and only if the invariant laminations of $\psi$ are orientable.\qed
\end{theorem}
It is therefore clear that when the lamination $\mathcal{L}$ associated to a fibred face $\face{F}$ is transversely orientable, then the Teichm\"uller polynomial of $\face{F}$ and the Alexander polynomial of $N$ have to be very tightly related; for all cohomology classes \mbox{$\xi \in \mathrm{int}(\rr_+\hspace{-0.1cm}\cdot \face{F})\cap H^1(N;\zz)$} the largest real roots of their specialisations at $\xi$ have to be equal up to the sign. 

When the lamination $\mathcal{L}$ is not transversely orientable it is still possible that its intersection with a given fibre is an orientable 1-dimensional lamination. 
For brevity we say that cohomology classes determining fibrations with this property are \emph{orientable}. 

\begin{definition}
	Let $\face{F}$ be a fibred face of the Thurston norm ball in $H^1(N;\rr)$. Let $\mathcal{L}$ be the stable lamination of the suspension flow associated to $\face{F}$. Let $\xi \in \inter(\rr_+\hspace{-0.1cm}\cdot \face{F})\cap H^1(N;\zz)$. Denote by $S$ the unique taut representative of the Poincar\'e-Lefschetz dual of $\xi$. We say that the fibred class $\xi$ is \emph{orientable} if the 1-dimensional lamination $S \cap \mathcal{L}$ in $S$ is orientable. 
\end{definition}

Corollary \ref{cor:Teich_in_one} gives a formula relating $\Theta_{\mathtt{F}}$ and $\Delta_N$ in the case when the lamination associated to $\face{F}$ induces a transversely orientable lamination in the maximal free abelian cover of $N$. We will show that transverse orientability of this lamination is in fact equivalent to the existence of an orientable fibred class in $\rr_+\hspace{-0.1cm}\cdot \face{F}$. 

\begin{theorem}\label{thm:characterisation}
	Let $N$ be a compact, oriented, hyperbolic 3-manifold with a fibred face $\face{F}$ of the Thurston norm ball in $H^1(N;\rr)$. Let $\mathcal{L}$ be  the stable lamination of the suspension flow associated to $\face{F}$. If the lamination induced by $\mathcal{L}$ in the maximal free abelian cover of $N$ is transversely orientable then there is an orientable fibred class in the interior of $\rr_+ \hspace{-0.1cm}\cdot \face{F}$.
\end{theorem}
This theorem follows immediately from Corollary \ref{cor:orientable:classes} and Proposition \ref{prop:no:orientable:class} that we prove below.
 First note that if $N$ is fibred over the circle with fibre $S$ and monodromy~$\psi$ then
\begin{equation}\label{eqn:H_1N:splits}
	H_1(N;\zz) = \left(\bigslant{H_1(S;\zz)}{\lbrace \gamma = \psi_\ast(\gamma)\rbrace}\right) \oplus \zz.\end{equation}
Let $H_1(S;\zz)_\psi$ denote $\bigslant{H_1(S;\zz)}{\lbrace \gamma = \psi_\ast(\gamma)\rbrace}$ and let $H_S^\psi$ denote its torsion-free part. Then we have 
\begin{equation}\label{eqn:HN:splitting}
	H_N = H_S^\psi \oplus \zz.\end{equation}
\begin{proposition}\label{prop:less:useful:orientable:classes}
	Let~$\face{F}$ be a fibred face of the Thurston norm ball in $H^1(N;\rr)$. Let $\mathcal{L}$ be the stable lamination of the suspension flow associated to $\face{F}$. If $\mathcal{L}$ is not transversely orientable, but the induced lamination $\mathcal{L}^N \subset N^{fab}$ is transversely orientable,	 then  a primitive class  $\xi \in \inter(\rr_+\hspace{-0.1cm}\cdot \face{F}) \cap H^1(N;\zz)$ is orientable if and only if for every $h \in H_N$
	\begin{equation}\label{cond:orientable:all}
		\xi(h) = \begin{cases}
			\text{odd} &\text{if }\sigma_N(h) = -1\\
			\text{even} &\text{if } \sigma_N(h) = 1
		\end{cases}.\end{equation}
\end{proposition}
\begin{proof}Let $S$ be the taut representative of the Poincar\'e-Lefschetz dual of $\xi$. We get a splitting  $H_N = H_S^\psi \oplus \zz$ as in \eqref{eqn:HN:splitting}. Since we assume that $\mathcal{L}$ is not transversely orientable,  $\xi$ is orientable if and only if $H_S^\psi = \ker \sigma_N$. 
		
	Let $u \in H_N$ be the (Hom-)dual of $\xi$. A proof that the condition \eqref{cond:orientable:all} is sufficient and necessary follows from the observation that $\xi(h) = k$ if and only if  $h=h'+k\cdot u$ for some $h' \in H_S^\psi$, $k \in \zz$, and hence $\sigma_N(h)=\sigma_N(h')\cdot\sigma_N(u)^k$.
\end{proof}
To obtain an easy  recipe to find orientable fibred classes we can express the condition \eqref{cond:orientable:all} in terms of a fixed basis of $H_N$.
\begin{corollary}\label{cor:orientable:classes}
	Let~$\face{F}$ be a fibred face of the Thurston norm ball in $H^1(N;\rr)$. Let  $(h_1, \ldots, h_s)$ be a basis of $H_N$ and let $(h_1^\ast, \ldots, h_s^\ast)$ be the dual basis of $H^1(N;\zz)$. Denote by $\mathcal{L}$ the stable lamination associated to $\face{F}$.
	
	If $\mathcal{L}$ is not transversely orientable, but the induced lamination $\mathcal{L}^N \subset N^{fab}$ is transversely orientable
	then a primitive class
	\[\xi=a_1h_1^\ast + \ldots + a_r h_r^\ast \in \inter(\rr_+\hspace{-0.1cm}\cdot \face{F}) \cap H^1(N;\zz)\]
	is orientable if and only if for every $j=1, \ldots, s$
	\begin{equation}\label{eqn:orientable:classes}
		a_j = \begin{cases}
			\text{odd} &\text{if }\sigma_N(h_j) = -1\\
			\text{even} &\text{if }\sigma_N(h_j) = 1.
		\end{cases}
	\end{equation} In particular, with the above assumptions on $\mathcal{L}$ there are orientable fibred classes in $\rr_+ \hspace{-0.1cm}\cdot \face{F}$. \qed
\end{corollary}
\begin{remark*}
	For a fibred face of the Thurston norm ball of the ``magic'' manifold (s776 in the SnapPy census \cite{snappea}) this was proved by Kin and Takasawa \cite[Proposition~3.5]{Kin_Takasawa}.
\end{remark*} 

Using the formulas given in Corollary~\ref{cor:Teich_in_one} one can directly check that if  $\mathcal{L}$ is not transversely orientable, but~$\mathcal{L}^N$ is, then indeed $\xi(\Delta_{N})(-z)$ divides $\xi(\Theta_{\texttt{F}})(z)$ if and only if the condition \eqref{eqn:orientable:classes} is satisfied.


We will now show that when the lamination $\mathcal{L}^N$ in the maximal free abelian cover $N^{fab}$ is not transversely orientable, then no fibred class in the corresponding fibred cone is orientable.
\begin{proposition}\label{prop:no:orientable:class}
	Let $\face{F}$ be a fibred face of the Thurston norm ball in $H^1(N;\rr)$. Let $\mathcal{L}$ be the stable lamination of the suspension flow associated to $\face{F}$. If the induced lamination~$\mathcal{L}^N$ in $N^{fab}$ is not transversely orientable, then there is no orientable fibred class in the interior of $\rr_+\hspace{-0.1cm}\cdot \face{F}$.
\end{proposition}
\begin{proof}
	Pick $\xi \in \inter(\rr_+\hspace{-0.1cm}\cdot \face{F})\cap H^1(N;\zz)$. Let $M = N \bez \mathrm{sing}(\face{F})$ and denote by $\mathcal{V}$ the veering triangulation of $M$ determined by $\face{F}$.  Let $i^\ast: H^1(N;\rr) \rightarrow H^1(M;\rr)$ be the homomorphism induced by the inclusion of $M$ into $N$.	
	Denote by $S$ the fibre of the fibration of~$M$ determined by  $i^\ast \xi \in H^1(M;\zz)$, and by $\psi$ the monodromy of this fibration.	
	As in \eqref{eqn:H_1N:splits} we get a splitting \[H_1(M;\zz) = H_1(S;\zz)_\psi \oplus \zz.\] 
	
	The edge-orientation homomorphism $\omega$ of $\mathcal{V}$ always factors through \linebreak \mbox{$\overline{\omega}: H_1(M;\zz) \rightarrow \lbrace -1, 1 \rbrace$}.	The class $\xi$ is orientable if and only if $i^\ast\xi$ is orientable, that is if and only if $H_1(S;\zz)_\psi \leq \ker \overline{\omega}$.
	
	First suppose that  $\face{F}$ is fully-punctured. 
	By Corollaries \ref{cor:even_torsion} and \ref{cor:trans_or_edge_or}, if  $\mathcal{L}^N = \mathcal{L}^{fab}$ is not transversely orientable, then there is a torsion element in $H_1(S;\zz)_\psi$ with a nontrivial image under $\overline{\omega}$. 
	
	If $\face{F}$ is not fully-punctured then the lamination $\mathcal{L}^N$ in~$N^{fab}$ is not transversely orientable either because $\mathcal{L}^{fab}$ in $M^{fab}$ is not transversely orientable, or because at least one of the Dehn filling slopes that recover $N$ from $M$ has a nontrivial image under $\overline{\omega}$. In either case we obtain an element of $H_1(S;\zz)_\psi$ with a nontrivial image under $\overline{\omega}$.	
\end{proof}
Proposition \ref{prop:no:orientable:class} and Theorem \ref{thm:orientable:lam} imply that if the lamination $\mathcal{L}^N$ in $N^{fab}$ is not transversely orientable, then the largest in the absolute value real roots of $\xi(\Delta_N)$ and $\xi(\Theta_{\mathtt{F}})$ are different for every fibred class $\xi \in \rr_+\hspace{-0.1cm}\cdot \face{F}$. When the face $\face{F}$ is fully-punctured the relation between $\Theta_{\mathtt{F}}$ and $\Delta_N$ is given in Proposition \ref{prop:twisted_factors}.

\color{black}
	\bibliographystyle{abbrv}
	\bibliography{mybib}

\begin{thebibliography}{10}

\bibitem{oriented-cover}
S.~Adams-Florou.
\newblock Oriented cover.
\newblock {\em {T}he {M}anifold {A}tlas {P}roject}, 2014.
\newblock \url{http://www.map.mpim-bonn.mpg.de/Oriented_cover#Hatcher2002l}.

\bibitem{Agol_veer}
I.~Agol.
\newblock Ideal triangulations of pseudo-{A}nosov mapping tori.
\newblock In W.~Li, L.~Bartolini, J.~Johnson, F.~Luo, R.~Myers, and J.~H.
  Rubinstein, editors, {\em Topology and Geometry in Dimension Three:
  Triangulations, Invariants, and Geometric Structures}, volume 560 of {\em
  Contemporary Mathematics}, pages 1--19. American Mathematical Society, 2011.

\bibitem{Tsang-Agol}
I.~Agol and C.~C. Tsang.
\newblock Dynamics of veering triangulations: infinitesimal components of their
  flow graphs and applications.
\newblock \texttt{arXiv:2201.02706v1 [math.GT]}, January 2022.

\bibitem{BandBoyland}
G.~Band and P.~Boyland.
\newblock The {B}urau estimate for the entropy of a braid.
\newblock {\em Algebr. Geom. Topol.}, 7:1345--1378, 2007.

\bibitem{crow_fox}
R.~Crowell and R.~Fox.
\newblock {\em Introduction to {K}not {T}heory}, volume~57 of {\em Graduate
  Texts in Mathematics}.
\newblock Springer-Verlag New York, 1963.

\bibitem{snappea}
M.~Culler, N.~Dunfield, and J.~R. Weeks.
\newblock Snap{P}y. {A} computer program for studying the geometry and topology
  of 3-manifolds.
\newblock {\tt http://snappy.computop.org/}.

\bibitem{Dunfield_Alex}
N.~R. Dunfield.
\newblock Alexander and {T}hurston norms of fibered 3-manifolds.
\newblock {\em Pacific Journal of Mathematics}, 200(1):43--58, 2001.

\bibitem{Fenley_geometry}
S.~R. Fenley.
\newblock Foliations, topology and geometry of 3-manifolds:
  $\mathbb{R}$-covered foliations and transverse pseudo-{A}nosov flows.
\newblock {\em Comment. Math. Helv.}, 77:415--490, 2002.

\bibitem{FoxCalculus1}
R.~H. Fox.
\newblock Free {D}ifferential {C}alculus. {I}: {D}erivation in the {F}ree
  {G}roup {R}ing.
\newblock {\em Annals of Mathematics}, 57(3):547--560, 1953.

\bibitem{FoxCalculus5}
R.~H. Fox.
\newblock Free {D}ifferential {C}alculus. {V}: {T}he {A}lexander {M}atrices
  {R}e-examined.
\newblock {\em Annals of Mathematics}, 71(3):408--422, 1960.

\bibitem{SchleimSegLinks}
S.~Frankel, S.~Schleimer, and H.~Segerman.
\newblock From veering triangulations to link spaces and back again.
\newblock arXiv:1911.00006 [math.GT].

\bibitem{Fried_suspension}
D.~Fried.
\newblock Fibrations over ${S}^1$ with {P}seudo-{A}nosov {M}onodromy.
\newblock In A.~Fathi, F.~Laudenbach, and V.~Po\'enaru, editors, {\em
  Thurston's work on surfaces}, chapter~14, pages 215--230. Princeton
  University Press, 2012.

\bibitem{FriedlVidussi_survey}
S.~Friedl and S.~Vidussi.
\newblock A survey of twisted {A}lexander polynomials.
\newblock In M.~Banagl and D.~Vogel, editors, {\em The {M}athematics of
  {K}nots}, volume~1 of {\em Contributions in {M}athematical and
  {C}omputational {S}ciences}, pages 45--94. Springer, Berlin, Heidelberg,
  2011.

\bibitem{explicit_angle}
D.~Futer and F.~Gu\'eritaud.
\newblock Explicit angle structures for veering triangulations.
\newblock {\em Algebr. Geom. Topol.}, 13(1):205--235, 2013.

\bibitem{FuterTaylorWorden}
D.~Futer, S.~Taylor, and W.~Worden.
\newblock Random veering triangulations are not geometric.
\newblock \texttt{arXiv:1808.05586 [math.GT]}.

\bibitem{VeeringCensus}
A.~Giannopolous, S.~Schleimer, and H.~Segerman.
\newblock A census of veering structures.
\newblock \url{https://math.okstate.edu/people/segerman/veering.html}.

\bibitem{Gueritaud_CT}
F.~Gu\'eritaud.
\newblock Veering triangulations and {C}annon-{T}hurston maps.
\newblock {\em Journal of Topology}, 9(3):957--983, 2016.

\bibitem{veer_strict-angles}
C.~D. Hodgson, J.~H. Rubinstein, H.~Segerman, and S.~Tillmann.
\newblock Veering triangulations admit strict angle structures.
\newblock {\em Geometry \& Topology}, 15(4):2073--2089, 2011.

\bibitem{Kin_Takasawa}
E.~Kin and M.~Takasawa.
\newblock Pseudo-{A}nosovs on closed surfaces having small entropy and the
  whitehead sister link exterior.
\newblock {\em J. Math. Soc. Japan}, 65(2):411--446, 2013.

\bibitem{Lack_taut}
M.~Lackenby.
\newblock Taut ideal triangulations of 3-manifolds.
\newblock {\em Geometry \& Topology}, 4(1):369--395, 2000.

\bibitem{Landry_stable}
M.~Landry.
\newblock Stable loops and almost transverse surfaces.
\newblock \texttt{arXiv:1903.08709 [math.GT].} To appear in Groups, Geometry,
  and Dynamics.

\bibitem{Landry_branched}
M.~Landry.
\newblock Taut branched surfaces from veering triangulations.
\newblock {\em Algebraic \& Geometric Topology}, 18(2):1089--1114, 2018.

\bibitem{Landry_homology_isotopy}
M.~Landry.
\newblock Veering triangulations and the {T}hurston norm: homology to isotopy.
\newblock {\em Advances in Mathematics}, 396, 2022.
\newblock \texttt{arXiv:2006.16328 [math.GT]}.

\bibitem{LMT_flow}
M.~Landry, Y.~N. Minsky, and S.~J. Taylor.
\newblock Flows, growth rates, and the veering polynomial.
\newblock \texttt{arXiv:2107.04066 [math.GT]}. To appear in Ergodic Theory and
  Dynamical Systems.

\bibitem{LMT}
M.~Landry, Y.~N. Minsky, and S.~J. Taylor.
\newblock A polynomial invariant for veering triangulations.
\newblock \texttt{arXiv:2008.04836 [math.GT]}.

\bibitem{McMullen_Teich}
C.~T. McMullen.
\newblock Polynomial invariants for fibered 3-manifolds and {T}eichm{\"u}ller
  geodesics for foliations.
\newblock {\em Ann. Scient. {\'E}c. Norm. Sup.}, 33(4):519--560, 2000.

\bibitem{McMullen_Alex}
C.~T. McMullen.
\newblock The {A}lexander polynomial of a $3$-manifold and the {T}hurston norm
  on cohomology.
\newblock {\em Annales scientifiques de l'\'Ecole Normale Sup\'erieure}, Ser.
  4, 35(2):153--171, 2002.

\bibitem{Milnor_duality}
J.~W. Milnor.
\newblock A duality theorem for {R}eidemeister torsion.
\newblock {\em Annals of Mathematics}, 76(1):137--147, 1962.

\bibitem{Milnor_covers}
J.~W. Milnor.
\newblock Infinite cyclic coverings.
\newblock In {\em Conference on the Topology of Manifolds (Michigan State
  Univ.)}, pages 115--133, 1968.

\bibitem{MinskyTaylor}
Y.~N. Minsky and S.~J. Taylor.
\newblock Fibered faces, veering triangulations, and the arc complex.
\newblock {\em Geom. Funct. Anal.}, 27(6):1450--1496, 2017.

\bibitem{Northcott}
D.~Northcott.
\newblock {\em Finite Free Resolutions}.
\newblock Cambridge Tracts in Mathematics. Cambridge University Press, 1976.

\bibitem{taut_veer_teich}
A.~Parlak.
\newblock Computation of the taut, the veering and the {T}eichm{\"u}ller
  polynomials.
\newblock {\em Experimental Mathematics}, 2021.
\newblock \texttt{https://doi.org/10.1080/10586458.2021.1985656}.

\bibitem{VeeringGitHub}
A.~Parlak, S.~Schleimer, and H.~Segerman.
\newblock Git{H}ub {V}eering repository. {R}egina-{P}ython and sage code for
  working with transverse taut and veering ideal triangulations.
\newblock \url{https://github.com/henryseg/Veering}.

\bibitem{HillmanSilverWilliams}
J.~A. Silver, D.~S. Silver, and S.~G. Williams.
\newblock On reciprocality of twisted alexander invariants.
\newblock {\em Algebr. Geom. Topol.}, 10(2):1017--1026, 2010.

\bibitem{Thur_norm}
W.~P. Thurston.
\newblock A norm for the homology of 3-manifolds.
\newblock {\em Memoirs of the American Mathematical Society}, 59(339):100--130,
  1986.

\bibitem{Thurston_map_tor}
W.~P. Thurston.
\newblock Hyperbolic structures on 3-manifolds, {II}: {S}urface groups and
  3-manifolds which fiber over the circle.
\newblock arXiv:math/9801045 [math.GT], 1998.

\bibitem{Traldi}
L.~Traldi.
\newblock The determinantal ideals of link modules.
\newblock {\em Pacific J. Math}, 101(1):215--222, 1982.

\bibitem{Tsang-Birkhoff}
C.~C. Tsang.
\newblock Constructing {B}irkhoff sections for pseudo-{A}nosov flows with
  controlled complexity.
\newblock \texttt{arXiv:2206.09586v1 [math.DS]}, June 2022.

\bibitem{Turaev_Alex}
V.~G. Turaev.
\newblock The {A}lexander polynomial of a three-dimensional manifold.
\newblock {\em Math. USSR Sbornik}, 26(3):313--329, 1975.

\bibitem{Worden_statistics}
W.~Worden.
\newblock Experimental statistics of veering triangulations.
\newblock {\em Experimental Mathematics}, 29(1):101--122, 2020.

\end{thebibliography}
\end{document}